\documentclass[a4paper,reqno]{amsart}
\usepackage{geometry}
\title[A refinement of the Pontryagin--Thom theorem]{A refinement of the Pontryagin--Thom theorem and its applications}
\author{Naoki Kuroda}
\address{GRADUATE SCHOOL OF MATHEMATICAL SCIENCES, THE UNIVERSITY OF TOKYO, 3-8-1 KOMABA, MEGURO-KU, TOKYO, 153-8914, JAPAN}
\email{kuronao0402@g.ecc.u-tokyo.ac.jp}
\usepackage{amsmath}
\usepackage{amssymb}
\usepackage{amsfonts}
\usepackage{mathrsfs}
\usepackage{amsthm}
\usepackage{extarrows}
\usepackage{bm}
\usepackage[dvipsnames]{xcolor}
\usepackage{graphicx}
\usepackage{wrapfig}
\usepackage[numbers]{natbib}
\usepackage{color}
\usepackage[labelsep=colon]{caption}
\usepackage[labelformat=empty,subrefformat=parens]{subcaption}
\usepackage{multicol}
\usepackage{tikz-cd}
\usepackage{booktabs}
\usepackage{hyperref}
\hypersetup{hidelinks}
\usepackage{cleveref}
\usepackage{pgfplots}
\pgfplotsset{compat=1.15}
\usepackage{tikz}
\usetikzlibrary{arrows.meta, decorations.pathmorphing, positioning}

\newtheorem{thm}{Theorem}[section]
\newtheorem{prop}[thm]{Proposition}
\newtheorem{lemma}[thm]{Lemma}
\newtheorem{cor}[thm]{Corollary}

\newtheorem*{mthm*}{Main Theorem}
\newtheorem*{thm*}{Theorem}
\newtheorem*{prop*}{Proposition}
\theoremstyle{definition}
\newtheorem{defi}[thm]{Definition}

\newtheorem{remark}[thm]{Remark}

\theoremstyle{plain}

\newcommand{\Spin}{\mathrm{Spin}}
\newcommand{\Pin}{\mathrm{Pin}}

\newcommand{\OrLine}[1][n]{\Omega_{#1}^{\mathrm{SO,line}}}
\newcommand{\SpinLine}[1][n]{\Omega_{#1}^{\mathrm{Spin,line}}}
\newcommand{\StringLine}[1][n]{\Omega_{#1}^{\mathrm{String,line}}}
\newcommand{\SpinPlane}[1][n]{\Omega_{#1}^{\mathrm{Spin,or\text{-}plane}}}
\newcommand{\StringPlane}[1][n]{\Omega_{#1}^{\mathrm{String,or\text{-}plane}}}
\newcommand{\String}{\mathrm{String}}
\newcommand{\Tring}{\mathrm{Tring}^{-}}
\newcommand{\Stringc}{\mathrm{String}^{c}}
\newcommand{\st}{\mathrm{st}}

\begin{document}

\begin{abstract}
We refine the Pontryagin--Thom theorem by showing that, for a real vector bundle $\xi\to X$ of rank $m>n\geq0$, the bordism group of closed smooth $n$-manifolds $M$ equipped with a map $g\colon M\to X$ and an actual (rather than stable) bundle isomorphism $TM\oplus\mathbb R^{m-n}\cong g^*\xi$ is naturally isomorphic to $\pi_{n-m}(X^{-\xi})$. Direct bordism is itself an equivalence relation except when $(n,m)=(1,2)$, where symmetry can fail.

Using Smith fiber sequences, we compute the kernels of natural surjections from bordism groups of manifolds $M$ with specified line subbundles of $TM\oplus\mathbb R$ or oriented two-plane subbundles of $TM\oplus\mathbb R^2$ to direct sums of ordinary bordism groups. These computations reduce to the parity of Euler characteristics. As applications, we determine when spheres bound manifolds equipped with tangent line fields that are normal along the boundary, and classify the Spin pseudo-Riemannian bordism groups introduced by Bais, May Custodio, and Torres.
\end{abstract}

\maketitle

\section{Introduction}

The Pontryagin--Thom construction is one of the fundamental links between differential topology and stable homotopy theory. It identifies bordism groups of manifolds equipped with prescribed tangential structures with homotopy groups of the corresponding Thom spectra \cite{Thom54,Rudyak98}.

In \cite[Theorem~A.2]{Ebert13}, Ebert recalls a modern formulation of the classical Pontryagin--Thom correspondence in terms of Thom spectra. In our notation, it takes the following form. Let $V$ be a virtual real vector bundle of rank $n\geq0$ over $X$. Then there is a natural group isomorphism
\[
\Omega_n^{\st}(X;V)\xrightarrow{\ \cong\ }\pi_0(X^{-V}).
\]
Here $\Omega_n^{\st}(X;V)$ is the bordism group of triples $(M,g,\Phi)$ consisting of a closed smooth $n$-manifold $M$, a continuous map $g\colon M\to X$, and a \textbf{stable} vector bundle isomorphism
\[
\Phi\colon TM\xrightarrow{\ \cong_{\st}\ }g^*V.
\]
The bordism manifolds are required to carry stable vector bundle isomorphisms compatible with the structures on their boundaries.

Reinhart introduced a bordism relation stronger than ordinary oriented bordism \cite{Reinhart63}. It requires a nowhere-zero tangent vector field on the bordism manifold that agrees with the prescribed normal directions along the boundary. Using the Pontryagin--Thom correspondence above, Ebert showed that Reinhart's bordism group is naturally isomorphic to the zeroth homotopy group of the Madsen--Tillmann spectrum $MTSO(n)$ \cite[Theorem~A.1]{Ebert13}.

Another instance in which unstabilized bundle isomorphisms are retained is Koschorke's nonstabilized Nielsen coincidence invariant \cite{Koschorke06}. By retaining the embedding of the coincidence locus and an unstabilized normal bundle isomorphism, his theory detects obstructions that are lost under stabilization. In contrast, the present paper does not include an embedding as part of the bordism data, but considers actual isomorphisms involving tangent bundles and the corresponding bordisms. Although the two bordism relations are different, both approaches retain geometric data before stabilization.

In this paper, we replace stable vector bundle isomorphisms with \textbf{actual} ones. By an actual vector bundle isomorphism, we mean an isomorphism between the specified bundles themselves, without adding further trivial summands. This gives a description of the homotopy groups of Thom spectra in terms of more concrete tangential data. Throughout the paper, spaces are assumed to have the homotopy type of CW complexes, and vector bundles are assumed to be numerable.

More precisely, let $n,m$ be integers with $0\leq n<m$, and let $\xi\to X$ be a real vector bundle of rank $m$. Applying the above Pontryagin--Thom correspondence to $V=\xi-\mathbb R^{m-n}$, we obtain
\[
\Omega_n^{\st}(X;\xi-\mathbb R^{m-n})\xrightarrow{\ \cong\ }\pi_{n-m}(X^{-\xi}).
\]
The starting point of this paper is the question of whether replacing the stable vector bundle isomorphisms in the bordism group on the left with actual ones yields the same bordism group.

Let $\Omega_n(X;\xi)$ denote the set of bordism classes of triples $(M,g,\phi)$, where $M$ is a closed smooth $n$-manifold, $g\colon M\to X$ is a continuous map, and
\[
\phi\colon TM\oplus\mathbb R^{m-n}\xrightarrow{\ \cong\ }g^*\xi
\]
is an actual vector bundle isomorphism. A direct bordism from $(M_0,g_0,\phi_0)$ to $(M_1,g_1,\phi_1)$ consists of a compact smooth $(n+1)$-manifold $W$ with boundary $M_0\sqcup M_1$, a continuous map $h\colon W\to X$ extending $g_0$ and $g_1$, and an actual vector bundle isomorphism
\[
\psi\colon TW\oplus\mathbb R^{m-n-1}\xrightarrow{\ \cong\ }h^*\xi
\]
whose boundary restrictions agree with $\phi_0$ and $\phi_1$ under a prescribed convention for the boundary normal directions.

Note that direct bordism need not itself be an equivalence relation. In particular, because of the convention for the boundary normal directions, the existence of a direct bordism from $M_0$ to $M_1$ does not in general imply the existence of one in the opposite direction. We therefore define $\Omega_n(X;\xi)$ using the equivalence relation generated by direct bordism. See Definition~\ref{def:cobordism_group_A} for the precise definition and boundary conditions.

Our first main result shows that the bordism group defined using actual vector bundle isomorphisms is naturally isomorphic to the corresponding homotopy group of the Thom spectrum.

\begin{thm}\label{thm:unstable-pt}
Let $n,m$ be integers with $0\leq n<m$, and let $\xi\to X$ be a real vector bundle of rank $m$. Stabilization induces a natural group isomorphism
\[
\Omega_n(X;\xi)\xrightarrow{\ \cong\ }\pi_{n-m}(X^{-\xi}).
\]

\end{thm}

This isomorphism allows us to study homotopy groups of Thom spectra using geometric data, such as subbundles of $TM\oplus\mathbb R^{m-n}$, that cannot readily be recovered from a stable isomorphism alone.

We also investigate when direct bordism is itself an equivalence relation. For oriented or unoriented bordisms equipped with vector fields, B\"okstedt--Svane showed in \cite[Theorem~4.4]{BS14} that direct bordism is symmetric if $d$ is odd or $r<d/2$, where $d$ is the dimension of the bordism manifold and $r$ is the number of vector fields. They also give a counterexample to symmetry for $d=2$ and $r=1$ in \cite[Example~4.6]{BS14}. For direct bordisms defined with respect to an arbitrary vector bundle $\xi\to X$, we show that symmetry holds whenever $(n,m)\neq(1,2)$. The exceptional case corresponds to this counterexample.

\begin{prop}\label{prop:direct-bordism}
Let $0\leq n<m$, and let $\xi\to X$ be a real vector bundle of rank $m$. If $(n,m)\neq(1,2)$, direct bordism in Definition~\ref{def:cobordism_group_A} is an equivalence relation. In particular, it is unnecessary to pass to the equivalence relation generated by direct bordism in the definition of $\Omega_n(X;\xi)$. On the other hand, when $(n,m)=(1,2)$, there are examples in which direct bordism fails to be symmetric.
\end{prop}

The unstable Pontryagin--Thom theorem allows us to interpret homotopy groups of Thom spectra as bordism groups of manifolds with actual tangential structures. Our next aim is to use this correspondence to compute specific bordism groups. The main tool for these computations is the Smith fiber sequence.

Let $W\to X$ be a real vector bundle and let $V$ be a virtual vector bundle over $X$. If $p\colon S(W)\to X$ denotes the unit sphere bundle projection, the Smith fiber sequence gives a fiber sequence of Thom spectra
\[
S(W)^{p^*V}\longrightarrow X^V\longrightarrow X^{V\oplus W}
\]
\cite[Theorem~5.1]{Detal26}. We organize the Smith fiber sequences into two types. The first consists of the three families $SO/O$, $\Spin/\Pin^-$, and $\String/\Tring$, which use real determinant line bundles. The second consists of the two families $\Spin/\Spin^c$ and $\String/\Stringc$, which use complex determinant line bundles regarded as oriented real two-plane bundles. For these five families, we construct commutative $3\times3$ diagrams in two common forms, with rank stabilization in one direction and change of structure group in the other. The classifying spaces for the String refinements are defined as homotopy pullbacks, so that a null-homotopy of $\lambda=\frac12p_1$ is recorded as part of the structure. We use the resulting long exact sequences to study five geometric bordism groups.

For $n\geq1$, we write $\OrLine$, $\SpinLine$, and $\StringLine$ for the bordism groups of closed oriented, Spin, and String $n$-manifolds, respectively, equipped with a specified line subbundle of $TM\oplus\mathbb R$. Similarly, $\SpinPlane$ and $\StringPlane$ denote the bordism groups of closed Spin and String $n$-manifolds equipped with a specified oriented two-plane subbundle of $TM\oplus\mathbb R^2$. The bordism manifolds are required to carry the corresponding tangential structures and subbundles. Precise definitions are given in Section~\ref{sec:geometric-groups}.

Before stating our results, we introduce notation for the universal vector bundles. For a positive integer $n$ and $G\in\{O,SO,\Spin,\Pin^-,\Spin^c\}$, let $L_n^G\to BG(n)$ be the real vector bundle of rank $n$ associated with the standard representation $G(n)\to O(n)$, and let $\det L_n^G=\Lambda^nL_n^G$ denote its determinant line bundle. Let $\mathcal O_n\to B\Spin^c(n)$ denote the complex line bundle associated with the determinant homomorphism $\Spin^c(n)\to U(1)$. We regard $\mathcal O_n$ as an oriented real vector bundle of rank two and use the same notation for its pullback to $B\Stringc(n)$. We write $L_n^{\String}$ for the pullback of $L_n^{\Spin}$ along $B\String(n)\to B\Spin(n)$. The universal real vector bundles $L_n^{\Tring}$ and $L_n^{\Stringc}$ appearing in the String refinements are the pullbacks along $B\Tring(n)\to B\Pin^-(n)$ and $B\Stringc(n)\to B\Spin^c(n)$, respectively. The classifying spaces and conventions for String structures are defined precisely in Section~\ref{sec:smith}.

In all five families, we consider the following geometric homomorphism. Let $(M,\eta)$ be a manifold equipped with a specified subbundle, and let $s$ be a section of $\eta$ transverse to the zero section. Write $0_M\subset\eta$ for the image of the zero section and $Z(s):=s^{-1}(0_M)=\{x\in M\mid s(x)=0_x\}$ for its zero locus. Then $Z(s)$ is a closed submanifold of $M$, with normal bundle naturally isomorphic to $\eta|_{Z(s)}$. The structures on $M$ and $\eta$ induce a tangential structure on $Z(s)$. We denote by $\epsilon$ the homomorphism that records the bordism classes of this zero locus and of $M$ itself. Ordinary bordism groups in negative degrees are understood to be zero.

We first state the results for the three line-type families $SO/O$, $\Spin/\Pin^-$, and $\String/\Tring$, followed by the two plane-type families $\Spin/\Spin^c$ and $\String/\Stringc$.

\begin{thm}\label{thm:oriented-line}
Let $n\geq1$. There are natural group isomorphisms
\[
\OrLine
\cong
\Omega_n(BO(n);L_n^O\oplus\det L_n^O)
\cong
\pi_{-1}\bigl(BO(n)^{-(L_n^O\oplus\det L_n^O)}\bigr).
\]
Moreover, the homomorphism
\[
\epsilon\colon \OrLine\longrightarrow \Omega_{n-1}^O\oplus\Omega_n^{SO},
\qquad
[M,\eta]\longmapsto([Z(s)],[M])
\]
is surjective, and
\[
\ker\epsilon\cong
\begin{cases}
\mathbb Z/2,&n\equiv1\pmod4,\\
0,&n\not\equiv1\pmod4
\end{cases}.
\]

\end{thm}

\begin{thm}\label{thm:spin-line}
Let $n\geq1$. There are natural group isomorphisms
\[
\SpinLine
\cong
\Omega_n(B\Pin^-(n);L_n^{\Pin^-}\oplus\det L_n^{\Pin^-})
\cong
\pi_{-1}\bigl(B\Pin^-(n)^{-(L_n^{\Pin^-}\oplus\det L_n^{\Pin^-})}\bigr).
\]
Moreover, the homomorphism
\[
\epsilon\colon \SpinLine\longrightarrow \Omega_{n-1}^{\Pin^-}\oplus\Omega_n^{\Spin},
\qquad
[M,\eta]\longmapsto([Z(s)],[M])
\]
is surjective, and
\[
\ker\epsilon\cong
\begin{cases}
\mathbb Z/2,&n\equiv1,3,4,5\pmod8,\\
0,&n\equiv0,2,6,7\pmod8
\end{cases}.
\]

\end{thm}

\begin{thm}\label{thm:string-line}
Let $n\geq1$. There are natural group isomorphisms
\[
\StringLine
\cong
\Omega_n(B\Tring(n);L_n^{\Tring}\oplus\det L_n^{\Tring})
\cong
\pi_{-1}\bigl(B\Tring(n)^{-(L_n^{\Tring}\oplus\det L_n^{\Tring})}\bigr).
\]
Moreover, the homomorphism
\[
\epsilon\colon \StringLine\longrightarrow \Omega_{n-1}^{\Tring}\oplus\Omega_n^{\String},
\qquad
[M,\eta]\longmapsto([Z(s)],[M])
\]
is surjective, and
\[
\ker\epsilon\cong
\begin{cases}
0,&n\equiv0,2,6,14,15\pmod{16},\\
\mathbb Z/2,&\text{otherwise}
\end{cases}.
\]

\end{thm}

\begin{thm}\label{thm:spin-plane}
Let $n\geq1$. There are natural group isomorphisms
\[
\SpinPlane
\cong
\Omega_n(B\Spin^c(n);L_n^{\Spin^c}\oplus\mathcal O_n)
\cong
\pi_{-2}\bigl(B\Spin^c(n)^{-(L_n^{\Spin^c}\oplus\mathcal O_n)}\bigr).
\]
Moreover, the homomorphism
\[
\epsilon\colon \SpinPlane\longrightarrow \Omega_{n-2}^{\Spin^c}\oplus\Omega_n^{\Spin},
\qquad
[M,\eta]\longmapsto([Z(s)],[M])
\]
is surjective, and
\[
\ker\epsilon\cong
\begin{cases}
\mathbb Z,&n\text{ even},\\
\mathbb Z/2,&n\equiv3\pmod8,\\
0,&n\equiv1,5,7\pmod8
\end{cases}.
\]

\end{thm}

\begin{thm}\label{thm:string-plane}
Let $n\geq1$. There are natural group isomorphisms
\[
\StringPlane
\cong
\Omega_n(B\Stringc(n);L_n^{\Stringc}\oplus\mathcal O_n)
\cong
\pi_{-2}\bigl(B\Stringc(n)^{-(L_n^{\Stringc}\oplus\mathcal O_n)}\bigr).
\]
Moreover, the homomorphism
\[
\epsilon\colon \StringPlane\longrightarrow \Omega_{n-2}^{\Stringc}\oplus\Omega_n^{\String},
\qquad
[M,\eta]\longmapsto([Z(s)],[M])
\]
is surjective, and
\[
\ker\epsilon\cong
\begin{cases}
\mathbb Z,&n\text{ even},\\
\mathbb Z/2,&n\equiv3,7,9,11\pmod{16}\text{ or }n=5,\\
0,&\text{otherwise}.
\end{cases}
\]

\end{thm}

The computations for the two String refinements use a vanishing lemma for Wu classes of String bundles, in addition to known restrictions on the parity of Euler characteristics. In the $\Tring$ case, we use the String structure on $TM\oplus\det(TM)$ to determine, modulo $16$, the dimensions in which an odd Euler characteristic can occur. In the $\String^c$ case, we use the String structure on $TM\oplus L_{\mathbb R}$ and construct examples with odd Euler characteristic using a $12$-dimensional cubic hypersurface and a $20$-dimensional Cayley plane bundle. Both computations are based on the finite-rank definitions of the String refinements. Related computations of SKK groups using the parity of Euler characteristics appear in \cite{HSV25}; see also \cite{HSV26} for the Euler characteristics of Pin manifolds.

The kernel of $\epsilon$ measures the information that is not detected by the bordism classes of the manifold itself and of the zero locus of a section of the specified subbundle.

Finally, we give two applications of these computations. The first application uses the line-type bordism groups to determine, for each $H=SO,\Spin,\String$, when there exists a compact $(n+1)$-manifold with an $H$-structure, boundary $S^n$, and a tangent line subbundle agreeing with the normal line along the boundary. For $n\geq1$, the conditions are as follows:
\[
\begin{array}{c|l}
H&\text{Condition on }n\\\hline
SO&n\not\equiv1\pmod4\\
\Spin&n\equiv0,2,6,7\pmod8\\
\String&n\equiv0,2,6,14,15\pmod{16}
\end{array}
\]
Thus the computations of the kernels in the line-type families give a uniform solution to the existence problems for all three structures.

The second application concerns the pseudo-Riemannian bordism introduced by Bais, May Custodio, and Torres \cite{BCT24}. After choosing nowhere-zero tangent vector fields on the boundary manifolds, we describe the bordism relation using a bordism group of the type considered in our unstable Pontryagin--Thom theorem. We then use the Smith long exact sequence and the computation of the kernel for $\SpinPlane$ to classify the resulting bordism classes independently of these choices. For closed Spin manifolds whose connected components have Euler characteristic zero, the ordinary Spin bordism class suffices when $n\not\equiv3\pmod8$; when $n\equiv3\pmod8$, the Kervaire semi-characteristic is also required. In particular, this removes the assumption on the Kervaire semi-characteristic in dimensions $n\equiv1\pmod4$ in \cite[Theorem~B]{BCT24}.

The paper is organized as follows. In Section~\ref{sec:preliminaries}, we establish our conventions for tangential bordism and Thom spectra. Section~\ref{sec:proof-A} proves the unstable Pontryagin--Thom theorem and establishes symmetry of direct bordism. In Section~\ref{sec:smith}, we construct the Smith diagrams for the three line-type families and the two plane-type families, in that order. Section~\ref{sec:geometric-groups} computes the corresponding bordism groups in the same order. The final section discusses the sphere boundary problem for the line-type families and the application of the plane-type families to the bordism of Bais--May Custodio--Torres.

\section*{Acknowledgments}

The author is grateful to Takuya Sakasai for his support and valuable advice. This research was supported by FMSP and the WINGS Program at the University of Tokyo.

\section*{Use of AI}

The results of the original version of this paper were obtained independently by the author. In preparing the revised version, the author used OpenAI's ChatGPT to assist in developing mathematical arguments and constructing examples, particularly in the computations of the bordism groups associated with String structures in Section~5. The tool was also used to assist in the development of proofs in Sections~3 and~6 and in the organization, drafting, and editing of the manuscript. The author takes full responsibility for all mathematical statements, proofs, and other contents of this paper.

\section{Tangential bordism and Thom spectra}\label{sec:preliminaries}

In this section, we establish notation and conventions for virtual vector bundles, tangential bordism groups, and Thom spectra. In particular, we distinguish the usual stable tangential bordism groups, denoted by the superscript $\st$, from the bordism groups defined using actual tangential structures, for which this superscript is omitted. We also use $\mathbb R^k$ to denote the trivial real vector bundle of rank $k$ when the base space is understood.

Consider a virtual vector bundle over $X$ with a chosen presentation
\[
V=V^+-V^-\in KO^0(X)
\]
where $V^+,V^-\to X$ are real vector bundles. Its rank is defined by
\[
\operatorname{rank}(V)=\operatorname{rank}(V^+)-\operatorname{rank}(V^-).
\]
We denote the corresponding Thom spectrum by $X^V$.

We first define the usual tangential bordism group using stable vector bundle isomorphisms.

\begin{defi}\label{def:cobordism_group}
Let $V=V^+-V^-$ be a virtual vector bundle of rank $n\geq0$ over $X$. Let $\Omega_n^{\st}(X;V)$ denote the set of bordism classes of triples $(M,g,\Phi)$, where:
\begin{itemize}
\item $M$ is a closed smooth $n$-manifold;
\item $g\colon M\to X$ is a continuous map;
\item $\Phi$ is a stable vector bundle isomorphism
\[
\Phi\colon TM\xrightarrow{\ \cong_{\st}\ }g^*V.
\]

\end{itemize}
By a stable vector bundle isomorphism, we mean the stable homotopy class of an actual vector bundle isomorphism
\[
\Phi_k\colon TM\oplus g^*V^-\oplus\mathbb R^k\xrightarrow{\cong}g^*V^+\oplus\mathbb R^k
\]
for some integer $k\geq0$. Two bundle isomorphisms represent the same stable isomorphism if, after adding further trivial summands, they are homotopic through bundle isomorphisms. Thus $\Phi$ specifies a stable homotopy class of isomorphisms, rather than merely asserting that $TM$ and $g^*V$ are equal in $KO^0(M)$.

A direct bordism from $(M_0,g_0,\Phi_0)$ to $(M_1,g_1,\Phi_1)$ is a triple $(W,h,\Psi)$, where $W$ is a compact smooth $(n+1)$-manifold with boundary $\partial W=M_0\sqcup M_1$, $h\colon W\to X$ is a continuous map restricting to $g_i$ on $M_i$, and $\Psi$ is a stable vector bundle isomorphism represented by
\[
\Psi\colon TW\oplus h^*V^-\oplus\mathbb R^k\xrightarrow{\cong}h^*V^+\oplus\mathbb R^{k+1}
\]
whose boundary restrictions are compatible with $\Phi_0$ and $\Phi_1$ as stable isomorphisms. In the identification
\[
TW|_{M_i}\cong TM_i\oplus\mathbb R
\]
we use the outward normal along $M_0$ and the inward normal along $M_1$.

We take the equivalence relation generated by direct bordism as the bordism relation and define addition by disjoint union.
\end{defi}

A stable isomorphism between virtual vector bundles induces an isomorphism between the corresponding bordism groups. The usual Pontryagin--Thom construction yields the following tangential bordism theorem; see, for example, \cite[Chapter~IV]{Rudyak98}.

\begin{thm}[Pontryagin--Thom]\label{thm:Pontryagin-Thom for spectrum}
Let $V$ be a virtual vector bundle of rank $n\geq0$ over $X$. There is a natural group isomorphism
\[
\Omega_n^{\st}(X;V)\xrightarrow{\cong}\pi_0(X^{-V}).
\]
More explicitly, the homotopy class corresponding to a triple $(M,g,\Phi)$ is represented by the composite
\[
\mathbb S
\xrightarrow{\operatorname{PT}_M}
M^{-TM}
\xrightarrow{(g,\Phi)}
X^{-V}
\]
obtained by the Pontryagin--Thom construction.
\end{thm}

Here $\mathbb S$ denotes the sphere spectrum, $M^{-TM}$ is the Thom spectrum of the virtual bundle $-TM$, and the last map is induced by $g\colon M\to X$ and the stable isomorphism $\Phi$.

We now define the unstable version of the bordism group considered in this paper.

\begin{defi}\label{def:cobordism_group_A}
Let $0\leq n<m$, and let $\xi\to X$ be a real vector bundle of rank $m$. Let $\Omega_n(X;\xi)$ denote the set of bordism classes of triples $(M,g,\phi)$, where:
\begin{itemize}
\item $M$ is a closed smooth $n$-manifold;
\item $g\colon M\to X$ is a continuous map;
\item $\phi$ is an actual vector bundle isomorphism
\[
\phi\colon TM\oplus\mathbb R^{m-n}\xrightarrow{\cong}g^*\xi.
\]

\end{itemize}

A direct bordism from $(M_0,g_0,\phi_0)$ to $(M_1,g_1,\phi_1)$ is a triple $(W,h,\psi)$, where $W$ is a compact smooth $(n+1)$-manifold with boundary $\partial W=M_0\sqcup M_1$, $h\colon W\to X$ is a continuous map restricting to $g_i$ on $M_i$, and $\psi$ is an actual vector bundle isomorphism
\[
\psi\colon TW\oplus\mathbb R^{m-n-1}\xrightarrow{\cong}h^*\xi
\]
whose boundary restrictions agree with $\phi_i$ under the identifications
\[
TW|_{M_i}\oplus\mathbb R^{m-n-1}\cong TM_i\oplus\mathbb R^{m-n}.
\]
As above, we use the outward normal along $M_0$ and the inward normal along $M_1$, identifying the normal direction with the first trivial summand of $\mathbb R^{m-n}$.

Since direct bordism need not be symmetric in general, we use the equivalence relation generated by direct bordism. Addition is defined by disjoint union.
\end{defi}

The difference from ordinary stable bordism is that both the bundle isomorphisms in the triples and those required on the bordism manifolds are actual, rather than stable, isomorphisms. Thus the definition of $\Omega_n(X;\xi)$ imposes an apparently stronger geometric condition.

Stabilizing the actual vector bundle isomorphisms gives a natural homomorphism of commutative monoids
\[
\st\colon \Omega_n(X;\xi)\longrightarrow\Omega_n^{\st}(X;\xi-\mathbb R^{m-n}).
\]
Theorem~\ref{thm:unstable-pt} asserts that this map is an isomorphism. In particular, the commutative monoid $\Omega_n(X;\xi)$ under disjoint union is in fact a group.

Although Definition~\ref{def:cobordism_group_A} uses the equivalence relation generated by direct bordism, Proposition~\ref{prop:direct-bordism} shows that direct bordism is already an equivalence relation when $(n,m)\neq(1,2)$. In this range, no zigzag is needed: two triples represent the same class if and only if there is a single direct bordism between them.

\section{Proof of the unstable Pontryagin--Thom theorem}\label{sec:proof-A}

In this section, we prove Theorem~\ref{thm:unstable-pt} by removing extra trivial summands from stable isomorphisms. The case $m\geq n+2$ is handled by a stable-range argument. In the borderline case $m=n+1$, we localize the relative obstruction on spheres and eliminate it. We then give a second proof using cobordism categories, followed by the proof of symmetry in Proposition~\ref{prop:direct-bordism}.

Throughout this section, put $k=m-n\geq1$. Fix a fiber metric on $\xi$; whenever necessary, we use polar decomposition to replace bundle isomorphisms with isometries. This can be done by a homotopy and hence does not affect the bordism classes. Write
\[
\st\colon \Omega_n(X;\xi)\longrightarrow\Omega_n^{\st}(X;\xi-\mathbb R^k)
\]
for the stabilization map. By the Pontryagin--Thom theorem in Section~\ref{sec:preliminaries}, the target is naturally identified with $\pi_{-k}(X^{-\xi})=\pi_{n-m}(X^{-\xi})$. It therefore suffices to show that $\st$ is bijective.

\subsection{Stable-range cancellation}\label{subsec:stable-range}

We begin with a general statement on stabilization of vector bundle isomorphisms, including the relative version needed below. For real vector bundles $E$ and $F$ of the same rank, let $\operatorname{Iso}(E,F)$ denote the bundle of fiberwise linear isomorphisms and $\Gamma\operatorname{Iso}(E,F)$ its space of sections.

\begin{lemma}[Destabilization of bundle isomorphisms]\label{lem:destabilization}
Let $K$ be a $d$-dimensional CW complex and let $E,F\to K$ be real vector bundles of rank $r$. Stabilization induces a map
\[
\pi_0\Gamma\operatorname{Iso}(E,F)\longrightarrow\pi_0\Gamma\operatorname{Iso}(E\oplus\mathbb R,F\oplus\mathbb R).
\]
This map is surjective if $r\geq d+1$ and bijective if $r\geq d+2$. Moreover, let $(K,A)$ be a relative CW complex of dimension at most $d$, and suppose that a stable isomorphism over $K$ is given whose restriction to $A$ agrees with the stabilization of a specified actual isomorphism. If $r\geq d+1$, the stable isomorphism can be deformed, relative to $A$, to the stabilization of an actual isomorphism.
\end{lemma}

\begin{proof}
After choosing metrics and restricting to orthogonal isomorphisms, the problem reduces to obstruction theory for the fiber sequence $O(r)\to O(r+1)\to S^r$. Evaluating an isomorphism of the stabilized bundles on the unit vector of the added trivial line gives a section of the corresponding sphere bundle with fiber $S^r$. If this section can be deformed to the standard section, lifting the homotopy along the fiberwise version of the above sequence yields an isomorphism fixing the last trivial line. Since $S^r$ is $(r-1)$-connected, the obstructions to this deformation vanish for $d\leq r-1$. When an actual isomorphism is specified over $A$, we perform the same argument relative to $A$. The relative homotopy involves cells of dimension at most $d+1$, so the obstruction groups involve $\pi_j(S^r)$ with $j\leq d$ and vanish if $d<r$. Iteration removes any finite number of additional trivial summands. For injectivity, apply the relative statement to a homotopy over $K\times I$ with both ends fixed, which requires $r\geq d+2$.
\end{proof}

We distinguish two ranges. Over an $n$-manifold, a stable isomorphism between bundles of rank $m\geq n+1$ can be replaced by an actual one. Over an $(n+1)$-dimensional bordism manifold, however, doing so while fixing the boundary data requires $m\geq n+2$. The borderline case $m=n+1$ is special precisely because this relative stable-range inequality fails.

\subsection{Surjectivity and injectivity in the stable range}\label{subsec:easy-range}

We first prove surjectivity. Represent an element of $\Omega_n^{\st}(X;\xi-\mathbb R^k)$ by $(M,g,\Phi)$. By definition, the stable isomorphism $\Phi$ is represented, for some $N\geq0$, by an actual isomorphism
\[
\Phi\colon TM\oplus\mathbb R^{k+N}\xrightarrow{\cong}g^*\xi\oplus\mathbb R^N.
\]
The two bundles to be compared, $TM\oplus\mathbb R^k$ and $g^*\xi$, have rank $m\geq n+1$. Repeated application of Lemma~\ref{lem:destabilization} therefore yields an actual vector bundle isomorphism
\[
\phi\colon TM\oplus\mathbb R^k\xrightarrow{\cong}g^*\xi
\]
whose stabilization represents $\Phi$. Thus $\st$ is surjective.

We next prove injectivity when $m\geq n+2$, or equivalently $k\geq2$. Suppose that $\st([M_0,g_0,\phi_0])=\st([M_1,g_1,\phi_1])$. The two stabilized triples are connected by a finite zigzag of stable direct bordisms. At each intermediate object, choose an actual isomorphism representing the given stable isomorphism, as in the surjectivity argument. It suffices to treat each arrow separately, so fix one stable direct bordism $(W,h,\Psi)$. After further stabilization if necessary, adjust $\Psi$ in a collar of the boundary so that its restrictions agree exactly with the stabilizations of the prescribed actual boundary isomorphisms.

We seek an actual vector bundle isomorphism over $W$ of the form
\[
\psi\colon TW\oplus\mathbb R^{k-1}\xrightarrow{\cong}h^*\xi.
\]
Both bundles have rank $m$, and $\dim W=n+1<m$. By the relative version of Lemma~\ref{lem:destabilization}, applied to $(W,\partial W)$, we can destabilize $\Psi$ while keeping the boundary isomorphisms fixed. The resulting $\psi$ gives a direct bordism. Applying this argument to every arrow in the zigzag shows that the two triples represent the same element of $\Omega_n(X;\xi)$. Hence $\st$ is injective when $m\geq n+2$.

\subsection{Injectivity in the borderline case \texorpdfstring{$m=n+1$}{m=n+1}: geometric correction}\label{subsec:borderline}

It remains to consider the case $m=n+1$, or $k=1$. The bundles to be compared over a bordism manifold are now $TW$ and $h^*\xi$, both of rank $n+1=\dim W$. If every connected component of $W$ has nonempty boundary, then $W$ has the homotopy type of an $n$-dimensional CW complex, so an actual isomorphism can be obtained if the boundary conditions are disregarded. However, the relative CW pair $(W,\partial W)$ still has $(n+1)$-dimensional relative cells, and the relative version of Lemma~\ref{lem:destabilization} does not apply directly. There may thus be an obstruction to destabilizing the stable isomorphism while preserving the boundary data.

In this subsection, we reduce the obstruction to a local problem. By removing finitely many small disks from the interior of the bordism manifold, we can destabilize the isomorphism while retaining the boundary data. The new boundary components are copies of $S^n$ whose bundle isomorphisms are stably trivial relative to the standard disk framing. We first study these sphere data and show that they can be eliminated under the bordism relation generated by direct bordism.

Let $n\geq0$, fix $x\in X$, and choose an identification $\xi_x\cong\mathbb R^{n+1}$. For each point $v$ of the standard unit sphere $S^n\subset\mathbb R^{n+1}$, put
\[
\kappa_v\colon T_vS^n\oplus\mathbb R\xrightarrow{\cong}\mathbb R^{n+1},\qquad (u,t)\longmapsto u+tv.
\]
Given a map $A\colon S^n\to O(n+1)$, write $C(A)$ for the sphere equipped with the constant map $g\colon S^n\to X$ with value $x$ and the bundle isomorphism $\phi_v=A(v)\circ\kappa_v$. The map $A$ compares the given actual isomorphism with the standard disk framing. In particular, if $A$ is constant, then $C(A)$ admits a direct cap given by the standard disk.

For now, assume $n\geq1$. Fix the basepoint $v_0=e_{n+1}\in S^n$ and put
\[
R_v=I-2vv^{\mathsf T},\qquad J=R_{v_0},\qquad b(v)=JR_v,\qquad \beta=[b]\in\pi_nSO(n+1).
\]
Here $R_v$ is the reflection in the hyperplane perpendicular to $v$, and $b(v_0)=I$. Consider the fiber sequence
\[
SO(n+1)\longrightarrow SO(n+2)\longrightarrow S^{n+1}
\]
and the associated long exact sequence
\[
\pi_{n+1}SO(n+2)\longrightarrow\pi_{n+1}S^{n+1}\xrightarrow{\partial}\pi_nSO(n+1)\longrightarrow\pi_nSO(n+2).
\]
The connecting homomorphism $\partial$ gives the clutching class of the tangent bundle of $S^{n+1}$. With a suitable choice of generator of $\pi_{n+1}S^{n+1}\cong\mathbb Z$, its image is represented by the map $b$ defined above. It follows that
\[
\ker\left(\pi_nSO(n+1)\longrightarrow\pi_nSO(n+2)\right)=\langle\beta\rangle.
\]
By the stable range for the orthogonal groups, this is also the kernel of stabilization to $SO$. Hence any comparison map $A\colon S^n\to O(n+1)$ that is homotopic to a constant after stabilization is homotopic to $Qb^a$ for some constant matrix $Q\in O(n+1)$ and integer $a\in\mathbb Z$. Write $C_a:=C(b^a)$. When $\beta$ has finite order, the choice of $a$ need not be unique.

Next, we determine the effect of reversing the boundary normal. Under the standard isomorphism $\kappa_v$, reversing the normal direction amounts to replacing $A(v)$ by $A(v)R_v$. Write $A(v)\simeq Qb(v)^a$ as above. Since $R_v=Jb(v)$ and $Jb(v)J=b(v)^{-1}$, we have
\[
A(v)R_v\simeq Qb(v)^aJb(v)=QJb(v)^{1-a}.
\]
Thus, after normalization by a constant matrix, reversal of the normal direction acts on the stabilization kernel by
\begin{equation}\label{eq:reflection-action}
a\beta\longmapsto(1-a)\beta.
\end{equation}
In the nullbordism constructions below, constant matrices such as $Q$ can be absorbed by postcomposing the entire bundle isomorphism on each bordism with the same constant matrix.

We now use these observations to show that sphere data which become trivial after stabilization can be eliminated under our bordism relation.

\begin{lemma}[Sphere correction in the generated relation]\label{lem:sphere-correction}
Let $n\geq0$. If $A\colon S^n\to O(n+1)$ is homotopic to a constant map after stabilization, then $C(A)$ is equivalent to the empty manifold under the equivalence relation generated by direct bordism.
\end{lemma}

\begin{proof}
First suppose $n\geq1$. As above, $A$ is homotopic to $Qb^a$ for some $Q$ and $a$. A homotopy of bundle isomorphisms is realized by a cylinder, and postcomposition by $Q$ preserves direct bordisms of the constant-map data. It therefore suffices to show that $[C_a]=0$ for every $a\in\mathbb Z$.

Consider first $C_0$ and $C_1$. Trivialize the tangent bundle of the standard unit disk $D^{n+1}\subset\mathbb R^{n+1}$ by its coordinates. At a point $v\in S^n=\partial D^{n+1}$, the outward unit normal is $v$. Taking the boundary as the source gives the bundle isomorphism $\kappa_v(u,t)=u+tv$ and hence a direct bordism $C_0\to\varnothing$. Taking the boundary as the target instead uses the inward unit normal $-v$, giving the isomorphism $(u,t)\mapsto u-tv=R_v\kappa_v(u,t)$. Postcomposing the trivialization on the disk with the constant matrix $J$ gives the boundary comparison map $JR_v=b(v)$ and hence a direct bordism $\varnothing\to C_1$. Thus $[C_0]=[C_1]=0$ under the equivalence relation generated by direct bordism.

For arbitrary integers $a,c$, we construct a direct bordism from $C_{a+c}$ to $C_a\sqcup C_c$. Let
\[
P=D^{n+1}\setminus\left(\operatorname{int}D_1^{n+1}\sqcup\operatorname{int}D_2^{n+1}\right)
\]
be the manifold obtained by removing two disjoint round disks, centered at $p_1$ and $p_2$, from the interior of the standard unit disk $D^{n+1}$. Define $\rho_i\colon P\to S^n$ by $\rho_i(x)=(x-p_i)/\|x-p_i\|$. The map
\[
F\colon P\longrightarrow SO(n+1),\qquad F(x)=b(\rho_1(x))^a b(\rho_2(x))^c
\]
has restrictions homotopic to $b^{a+c}$ on the outer boundary and to $b^a$ and $b^c$ on the two inner boundaries, respectively, using radial parametrizations of the boundary spheres. Modify $F$ in collars so that its restrictions agree with these comparison maps, and postcompose the standard trivialization of $TP$ with $F$. If necessary, rescale the tangent directions in the collars by positive factors to account for the radii of the inner disks. Taking the outer boundary as the source and the inner boundaries as the targets, the prescribed normals are the outward normals of the corresponding disks in all three cases. Thus we obtain a direct bordism
\begin{equation}\label{eq:pants-relation}
C_{a+c}\longrightarrow C_a\sqcup C_c.
\end{equation}

Consequently, in the commutative monoid $\Omega_n(X;\xi)$ under disjoint union,
\[
[C_{a+c}]=[C_a]+[C_c].
\]
In particular, setting $a=-1$ and $c=1$ gives
\[
[C_0]=[C_{-1}]+[C_1].
\]
Since $[C_0]=[C_1]=0$, we have $[C_{-1}]=0$. Repeated application of the addition formula now gives $[C_a]=0$ for every integer $a$. This proves the assertion for $n\geq1$.

Finally, suppose $n=0$. The sphere $S^0$ consists of two points. The condition that $A\colon S^0\to O(1)$ is homotopic to a constant after stabilization means that its two values lie in the same component after stabilization. As $O(1)=\{\pm1\}$ and its two components remain distinct under stabilization, the values agree. The standard framing of an interval therefore gives a direct nullbordism.
\end{proof}

The preceding lemma shows that actual structures on spheres that become trivial after stabilization can be eliminated under the equivalence relation generated by direct bordism. We use this to remove the obstruction to relative destabilization of an arbitrary stable bordism. To do so, we must first check that the obstruction can be localized on finitely many spheres.

\begin{lemma}[Localization of the critical obstruction]\label{lem:relative-puncturing}
Let $m=n+1$. Suppose that $(W,h,\Psi)$ is a stable direct bordism with specified actual isomorphisms on its boundary. Then, after removing the interiors of finitely many pairwise disjoint disks in the interior of $W$, the resulting manifold $W^\circ$ admits an actual vector bundle isomorphism $TW^\circ\cong h^*\xi|_{W^\circ}$ restricting to the original boundary isomorphisms. Moreover, the comparison map on each new sphere boundary is homotopic to a constant map after stabilization.
\end{lemma}

\begin{proof}
Choose a representative of $\Psi$ which agrees exactly with the stabilizations of the prescribed actual isomorphisms in a collar of the boundary. Choose a finite relative triangulation of $(W,\partial W)$, and construct the destabilization inductively over its relative cells, keeping the actual boundary isomorphisms fixed.

Since $TW$ and $h^*\xi$ have rank $n+1$, the obstruction-theoretic argument of Lemma~\ref{lem:destabilization} applies to the relative cells of dimension at most $n$. Hence we can construct an actual isomorphism over the relative $n$-skeleton, together with a homotopy from its stabilization to $\Psi$. The obstruction groups for the remaining $(n+1)$-cells need not vanish. Instead of eliminating these top-dimensional obstructions, remove a small disk from the interior of each $(n+1)$-cell. Each punctured cell deformation retracts onto its outer boundary, so the actual isomorphism and the accompanying homotopy extend over it.

Let $W^\circ$ be the resulting manifold. It carries an actual isomorphism satisfying the original boundary conditions. Each new sphere $S^n$ bounds a removed disk in $W$, over which $h^*\xi$ admits a trivialization and $\Psi$ is defined. Since the stabilization of the actual isomorphism on $W^\circ$ is homotopic to $\Psi$, the comparison map on the new sphere is homotopic to a constant after stabilization. Furthermore, the restriction of $h$ to the sphere is homotopic to a constant by contracting the removed disk. Using the trivialization over that disk, this homotopy is realized by a cylinder with the specified bundle isomorphism, so the sphere data can be replaced by data of the form $C(A)$. The relative triangulation is finite, so only finitely many disks are required.
\end{proof}

We have thus localized the obstruction to relative destabilization in the case $m=n+1$ on finitely many sphere boundaries with stably trivial comparison maps. By Lemma~\ref{lem:sphere-correction}, these spheres can be eliminated under the chosen bordism relation. The two lemmas together yield the remaining injectivity of $\st$.

\begin{proof}[Completion of the geometric proof of Theorem~\ref{thm:unstable-pt}]
Surjectivity of $\st$ and injectivity when $m\geq n+2$ have already been established. It remains to prove injectivity for $m=n+1$.

Take two triples with actual structures whose images under $\st$ agree. They are connected by a zigzag of stable direct bordisms. As before, choose actual representatives of the stable isomorphisms at the intermediate objects. Applying Lemma~\ref{lem:relative-puncturing} to each arrow gives a bordism preserving the original boundary isomorphisms and having only finitely many additional sphere boundaries. Their comparison maps are homotopic to constants after stabilization, so Lemma~\ref{lem:sphere-correction} makes these sphere data equivalent to the empty manifold. Regarding the additional boundaries as targets, we conclude that the two original triples represent the same element of $\Omega_n(X;\xi)$. Thus $\st$ is injective also when $m=n+1$.

It follows that $\st$ is a bijective homomorphism of commutative monoids for all $0\leq n<m$. Combining this with the stable Pontryagin--Thom theorem gives natural group isomorphisms
\[
\Omega_n(X;\xi)\xrightarrow{\cong}\Omega_n^{\st}(X;\xi-\mathbb R^{m-n})\xrightarrow{\cong}\pi_{n-m}(X^{-\xi}).
\]
This completes the geometric proof of Theorem~\ref{thm:unstable-pt}.
\end{proof}

\subsection{Cobordism categories and a second proof}\label{subsec:cobordism-category}

We now describe the bordism group $\Omega_n(X;\xi)$ in terms of the path components of the classifying space of a cobordism category. To this end, we encode the last $k-1$ trivial directions in a frame bundle, allowing us to apply the Galatius--Madsen--Tillmann--Weiss theorem without the sphere correction used in the preceding subsection.

For $0\leq j\leq k-1$, let $F_j=\operatorname{Fr}_j(\xi)$ denote the bundle of orthonormal $j$-frames in $\xi$, with $F_0=X$. Let $p_j\colon F_j\to X$ be the projection and let $\eta_j$ be the orthogonal complement of the trivial subbundle spanned by the tautological $j$-frame. Then
\[
p_j^*\xi\cong\eta_j\oplus\mathbb R^j,\qquad
\operatorname{rank}\eta_j=m-j.
\]
In particular, $\eta_{k-1}$ is a real vector bundle of rank $n+1$, and we denote its classifying map by
\[
\theta\colon F_{k-1}\longrightarrow BO(n+1).
\]

Let $\mathcal C_\theta$ be the $(n+1)$-dimensional tangential cobordism category associated with this tangential structure, using the bundle-map model. Here $\mathbb R^\infty=\operatorname*{colim}_{N}\mathbb R^N$. An object consists of $a\in\mathbb R$, a closed smooth $n$-dimensional submanifold $M\subset\mathbb R^\infty$, a map $\widetilde g\colon M\to F_{k-1}$, and a vector bundle isomorphism
\[
TM\oplus\mathbb R\xrightarrow{\cong}\widetilde g^*\eta_{k-1}.
\]
A nonidentity morphism between objects at heights $a_0<a_1$ consists of an $(n+1)$-dimensional bordism $W$ embedded with collars in $[a_0,a_1]\times\mathbb R^\infty$, a map $\widetilde h\colon W\to F_{k-1}$, and a vector bundle isomorphism
\[
TW\xrightarrow{\cong}\widetilde h^*\eta_{k-1}
\]
compatible with the boundary structures. The structures are required to be products in the collars. The object and morphism spaces carry the usual topology of the embedded cobordism category.

Consider a bundle isomorphism $\phi\colon TM\oplus\mathbb R^k\cong g^*\xi$ representing an element of $\Omega_n(X;\xi)$. After deforming it to an isometry, take the images of the last $k-1$ standard basis vectors, in reverse order, as the frame, and restrict the isometry to $TM$ and the first trivial line to obtain the above $\theta$-structure. Conversely, the frame and the isomorphism onto its orthogonal complement recover the original structure up to homotopy. The analogous correspondence for bordisms is compatible with the boundary restrictions, after choosing product metrics near the boundary. The difference in the normal conventions is absorbed by reversing the time coordinate, which passes to the opposite category without changing its path components. Since homotopies of structures are realized by cylinders and the stabilized embedding spaces are weakly contractible, we obtain a natural identification of commutative monoids
\begin{equation}\label{eq:cob-category-pi0}
\Omega_n(X;\xi)\cong\pi_0B\mathcal C_\theta.
\end{equation}

The theorem of Galatius--Madsen--Tillmann--Weiss \cite[Section~5]{GTMW09} states that, for a tangential structure $\theta\colon B\to BO(d)$, taken to be a Serre fibration, there is a weak homotopy equivalence
\[
B\mathcal C_\theta\simeq\Omega^{\infty-1}B^{-\theta^*\gamma_d},
\]
where one may use a fibration replacement without changing the homotopy type of the Thom spectrum. Applying the theorem to the tangential structure above gives an isomorphism
\begin{equation}\label{eq:gmtw-identification}
\Omega_n(X;\xi)\cong\pi_{-1}(F_{k-1}^{-\eta_{k-1}}).
\end{equation}

We next compare the right-hand side with $\pi_{-k}(X^{-\xi})$ by forgetting the frame vectors one at a time. Since $F_{j+1}=S(\eta_j)$ and the pullback of $\eta_j$ to $F_{j+1}$ splits as $\eta_{j+1}\oplus\mathbb R$, we have a Smith fiber sequence
\begin{equation}\label{eq:smith-frames}
\Sigma^{-1}F_{j+1}^{-\eta_{j+1}}
\longrightarrow F_j^{-\eta_j}
\longrightarrow\Sigma^\infty(F_j)_+
\end{equation}
for $0\leq j\leq k-2$, whose first map is induced by the projection forgetting the last frame vector. Since $\Sigma^\infty(F_j)_+$ is connective, the first map induces a natural isomorphism for $q\leq-2$:
\[
\pi_{q+1}(F_{j+1}^{-\eta_{j+1}})
\xrightarrow{\cong}
\pi_q(F_j^{-\eta_j}).
\]
Iterating this isomorphism yields
\[
\pi_{-1}(F_{k-1}^{-\eta_{k-1}})
\xrightarrow{\cong}\cdots\xrightarrow{\cong}
\pi_{-k}(X^{-\xi}),
\]
where no iteration is needed when $k=1$.

Composing the above isomorphisms, we obtain
\begin{equation}\label{eq:gmtw-second-proof}
\Omega_n(X;\xi)\xrightarrow{\cong}\pi_{n-m}(X^{-\xi}).
\end{equation}
With the trivial directions ordered as in the original bundle isomorphism $\phi$, the naturality of the Pontryagin--Thom construction identifies this composite with the stabilization map $\st$ followed by the stable Pontryagin--Thom isomorphism (using compatible conventions for the time and suspension coordinates). This gives a second proof of Theorem~\ref{thm:unstable-pt}. The argument does not, however, establish symmetry of direct bordism itself, which we prove separately in the next subsection.

\subsection{When direct bordism is an equivalence relation}\label{subsec:direct-symmetry}

Finally, we prove Proposition~\ref{prop:direct-bordism}. Direct bordism is reflexive by the cylinder construction and transitive by gluing along collars. It remains to establish symmetry. Suppose that $(W,h,\psi)$ is a direct bordism from $M_0$ to $M_1$. To obtain a direct bordism in the opposite direction, we must reverse the normal directions on both boundary components. We seek to realize this reversal by modifying the bundle isomorphism over $W$.

First suppose $k=m-n\geq2$. Equip the rank-$m$ bundle $E=TW\oplus\mathbb R^{k-1}$ with a metric compatible with the boundary splitting, and prescribe the outward unit normal as a section over the boundary. Since $m>\dim W=n+1$, obstruction theory extends it to a unit section $u$ over all of $W$. Consider the reflection in the hyperplane perpendicular to $u$, given by
\[
R_u(z)=z-2\langle z,u\rangle u.
\]
Precomposing $\psi$ with $R_u$ gives a new actual isomorphism that reverses only the normal direction along the boundary and leaves the tangential and remaining trivial directions unchanged. This defines a direct bordism with source and target interchanged. Thus direct bordism is symmetric when $m\geq n+2$, regardless of the value of $n$.

Next suppose $k=1$, so $m=n+1$. Choose a tangent vector field $u$ on $W$ agreeing with the outward unit normal along $\partial W$ and perturb it in the interior to have finitely many isolated zeros. Removing small disks around the zeros gives a manifold $W^\circ$ on which $u$ is nowhere zero. Precomposing $\psi\colon TW\cong h^*\xi$ with the reflection in the hyperplane perpendicular to $u/\|u\|$ yields a bordism on $W^\circ$ with the original source and target interchanged. The new boundary components are finitely many copies of $S^n$, whose comparison maps are homotopic to constants after stabilization. Indeed, in local coordinates the reflection is $R_{u(v)/\|u(v)\|}$. After adding a trivial direction, the unit vector defining this reflection can be homotoped to the unit vector in that new direction. The stabilized reflection is therefore homotopic to a constant reflection. The original isomorphism extends over each removed disk, and, as in Lemma~\ref{lem:relative-puncturing}, the maps to $X$ on the new spheres can be made constant by cylindrical bordisms.

We therefore need a result stronger than Lemma~\ref{lem:sphere-correction}: the spheres must be eliminated by direct bordisms in a prescribed direction. The following lemma provides such bordisms.

\begin{lemma}[Direct caps for the spheres]\label{lem:direct-sphere-caps}
Let $n\geq0$ with $n\neq1$. If $A\colon S^n\to O(n+1)$ is homotopic to a constant map after stabilization, then there is a direct bordism from $C(A)$ to the empty manifold.
\end{lemma}

\begin{proof}
For $n=0$, the standard interval suffices. Suppose $n\geq1$. As in Lemma~\ref{lem:sphere-correction}, it suffices to construct direct caps for $C_a$ for all integers $a$: homotopies are realized by cylinders, and constant matrices can be absorbed into the trivializations of the caps. The sphere $C_0$ has the standard disk as a direct cap. To obtain a direct cap for $C_1$, restrict a parallelization of the torus $T^{n+1}$ to $T^{n+1}\setminus\operatorname{int}D^{n+1}$. The comparison map with the coordinate trivialization extends over the removed disk and hence has coefficient zero. Since the outward normal of the punctured torus is opposite to the outward normal of the disk, the normal reversal formula~\eqref{eq:reflection-action} gives coefficient $1$ on its boundary. Thus $C_1$ also admits a direct cap, after adjusting by a constant matrix and a cylinder if necessary.

If $n\geq2$ is even, then $2\beta=0$. Indeed, $b(-v)=b(v)$, and the antipodal map of $S^n$ has degree $-1$. Hence $\beta=-\beta$ (using the standard identification of based and free homotopy classes into the connected topological group $SO(n+1)$). Hence $C_a$ is homotopic, in its bundle isomorphism, to either $C_0$ or $C_1$, and both admit direct caps.

If $n\geq3$ is odd, we additionally construct a direct cap for $C_{-1}$. Put $d=n+1$ and consider the closed stably parallelizable manifold $P=S^2\times S^{d-2}$. Since $d\geq4$ is even, we have $\chi(P)=4$. The punctured manifold $P^\circ=P\setminus\operatorname{int}D^d$ has the homotopy type of a $(d-1)$-dimensional CW complex. Thus a stable trivialization of $TP$ restricts and destabilizes to a parallelization $\Phi\colon TP^\circ\xrightarrow{\cong}P^\circ\times\mathbb R^d$ whose stabilization is homotopic to the restriction of the chosen stable trivialization. Choose oriented coordinates on the removed disk and compatible metrics so that the comparison map takes values in $SO(n+1)$. After normalization at the basepoint by a constant matrix, denote this comparison map by $f\colon S^n\to SO(n+1)$. Since $f$ is null-homotopic after stabilization, write $[f]=a\beta$.

For the fixed unit vector $v_0\in\mathbb R^d$, define a nowhere-zero vector field on $P^\circ$ by $V(x)=\Phi_x^{-1}(v_0)$, taking the parallelization normalized as above. Extend it over $D^d$ so that it has only isolated zeros in the interior. By the Poincar\'e--Hopf theorem, the sum of their indices is $\chi(P)$. In the coordinates of the disk, the normalized boundary vector field is the map $v\mapsto f(v)^{-1}v_0$. Since the sum of the indices in the disk is the degree of this boundary map, we obtain
\[
\deg(v\mapsto f(v)^{-1}v_0)=\chi(P)=4.
\]
On the other hand,
\[
b(v)^{-1}v_0=2\langle v,v_0\rangle v-v_0,
\]
and this map has degree $2$: the regular value $v_0$ has preimages $\pm v_0$, both with local degree $+1$ because $n$ is odd. The map $SO(n+1)\to S^n$ given by $Q\mapsto Q^{-1}v_0$ induces a homomorphism on $\pi_n$. Thus $[f]=a\beta$ implies
\[
2a=\deg(v\mapsto f(v)^{-1}v_0)=4,
\]
so $a=2$.

Finally, the outward normal of $P^\circ$ is opposite to the outward normal of the removed disk. By the normal reversal formula~\eqref{eq:reflection-action}, the normalized boundary coefficient is therefore $(1-2)\beta=-\beta$. After a constant change of framing and a cylindrical bordism if necessary, this gives a direct cap for $C_{-1}$.

It remains to iterate~\eqref{eq:pants-relation} to obtain direct bordisms $C_a\to\bigsqcup_{i=1}^a C_1$ for $a>0$ and $C_a\to\bigsqcup_{i=1}^{-a}C_{-1}$ for $a<0$, and then to attach the corresponding caps. This completes the proof.
\end{proof}

\begin{proof}[Proof of Proposition~\ref{prop:direct-bordism}]
When $k=m-n\geq2$, reflection with respect to a nowhere-zero section gives symmetry of direct bordism for every $n$. When $k=1$ and $n\neq1$, regard all the additional sphere boundaries of $W^\circ$ as targets. By Lemma~\ref{lem:direct-sphere-caps}, each of them admits a direct cap, which can be attached to obtain a direct bordism with the original source and target interchanged. Thus direct bordism is symmetric when $(n,m)\neq(1,2)$. Reflexivity and transitivity follow from cylinders and gluing, respectively, so direct bordism is an equivalence relation in this range. The remaining counterexample is discussed below.
\end{proof}

\begin{remark}
A counterexample to symmetry for $(n,m)=(1,2)$ is given by B\"okstedt--Svane in \cite[Example~4.6]{BS14}. They consider bordisms equipped with vector fields for $d=2$ and $r=1$. Take a vector field with a single zero on a closed oriented surface $\Sigma_g$ of genus $g\geq2$. Removing a small disk containing the zero gives a compact surface $W=\Sigma_g\setminus\operatorname{int}D^2$ with boundary $S^1$ and a nowhere-zero tangent vector field. With the appropriate boundary convention, this gives a direct bordism from the empty manifold to a circle with the induced structure. If a direct bordism in the opposite direction existed, gluing the two bordisms along their boundary would give a nowhere-zero tangent vector field on a closed oriented surface of genus at least $g\geq2$, contradicting the Poincar\'e--Hopf theorem.

This counterexample also applies to our definition. Indeed, take $X=*$ and $\xi=\mathbb R^2$, so $m=n+1=2$. An actual structure on a bordism manifold is then precisely a parallelization $\psi\colon TW\xrightarrow{\cong}\mathbb R^2$. Since the surface $W$ above is oriented and carries a nowhere-zero tangent vector field, rotating the vector field through a right angle using a metric and the orientation yields a second vector field and hence a parallelization. Let $\phi$ be the induced boundary isomorphism, and consider the resulting object $(S^1,*,\phi)$. There is a direct bordism $\varnothing\to(S^1,*,\phi)$. If a direct bordism in the opposite direction existed, the matching boundary isomorphisms would give a parallelization on the closed surface obtained by gluing, leading to the same contradiction with its Euler characteristic. Therefore direct bordism need not be symmetric when $(n,m)=(1,2)$.
\end{remark}

\section{Smith fiber sequences}\label{sec:smith}

In this section, we divide the five families into the line-bundle families $(SO,O)$, $(\Spin,\Pin^-)$, and $(\String,\Tring)$, which use real determinant line bundles, and the oriented two-plane families $(\Spin,\Spin^c)$ and $(\String,\Stringc)$, which use complex determinant line bundles. For each pattern, we construct a $3\times3$ diagram of Smith fiber sequences, with change of structure group and rank stabilization in the two directions.

\begin{thm}[Smith fiber sequence]\label{thm:smith-general}
Let $X$ be a space of the homotopy type of a CW complex, $V$ a virtual real vector bundle over $X$, and $W\to X$ a finite-rank real vector bundle. If $p\colon S(W)\to X$ is the unit sphere bundle projection, there is a fiber sequence of Thom spectra
\[
S(W)^{p^*V}\longrightarrow X^V\longrightarrow X^{V\oplus W}
\]
as in \textup{\cite[Theorem~5.1]{Detal26}}. This fiber sequence is natural with respect to pullbacks of vector bundles.
\end{thm}

We call the second map in this fiber sequence the Smith map. If the total space of the unit sphere bundle is the classifying space of another tangential structure, the first term can also be expressed as the Thom spectrum associated with that structure. We make this identification for the line-bundle and two-plane patterns below. By a commutative diagram we mean a homotopy-commutative diagram of spectra.

We first establish notation for the five families, including the String refinements. Throughout this section, the rank $r$ is a positive integer. For $G\in\{O,SO,\Spin,\Pin^-,\Spin^c\}$, let $L_r^G\to BG(r)$ be the universal real vector bundle of rank $r$ associated with the standard representation. We denote its determinant line bundle by $\det L_r^G$ and the complex line bundle over $B\Spin^c(r)$ associated with the determinant homomorphism by $\mathcal O_r$. When used in the Smith construction, $\mathcal O_r$ is regarded as a naturally oriented real vector bundle of rank two.

Write $\lambda_r$ for the pullback of the characteristic class $\lambda=\frac12p_1$ on the stable classifying space $B\Spin$ to $B\Spin(r)$, and define
\[
B\String(r):=\operatorname{hofib}\bigl(B\Spin(r)\xrightarrow{\lambda_r}K(\mathbb Z,4)\bigr).
\]
Thus a String structure includes a choice of null-homotopy of $\lambda_r$. For a real vector bundle $E$ with a Pin$^-$ structure, there is a canonical Spin structure on $E\oplus\det E$, and for one with a Spin$^c$ structure, there is a canonical Spin structure on $E\oplus\mathcal O(E)$. Using the induced maps, we define
\[
\begin{aligned}
B\Tring(r)&:=B\Pin^-(r)\times^h_{B\Spin(r+1)}B\String(r+1),\\
B\Stringc(r)&:=B\Spin^c(r)\times^h_{B\Spin(r+2)}B\String(r+2)
\end{aligned}.
\]
Thus a $\Tring$ structure is a Pin$^-$ structure together with a String structure on $E\oplus\det E$, and a $\Stringc$ structure is a Spin$^c$ structure together with a String structure on $E\oplus\mathcal O(E)$. Throughout, $\Tring$ refers to Tring$^-$ in this sense. Our convention for $\Stringc$ is the strong level $-1$ convention, as distinguished from String$^c$ structures at other levels; see also \cite{DebrayYu} for this distinction.

We write $L_r^{\String}$, $L_r^{\Tring}$, and $L_r^{\Stringc}$ for the pullbacks of $L_r^{\Spin}$, $L_r^{\Pin^-}$, and $L_r^{\Spin^c}$, respectively. We also denote the pullback of $\mathcal O_r$ to $B\Stringc(r)$ by the same symbol.

In both patterns, let $q\in\{1,2\}$ be the rank of the specified subbundle and let $D_r\to BG(r)$ be the determinant bundle of rank $q$. As in the introduction, write $L_r^G\to BG(r)$ and $L_r^H\to BH(r)$ for the universal real vector bundles. The groups $G,H$ and the bundle $D_r$ are specified in the following table.
\[
\begin{array}{c|c|c|c|c}
\text{Pattern}&q&H&G&D_r\\ \hline
\text{Line}&1&SO&O&\det L_r^O\\
\text{Line}&1&\Spin&\Pin^-&\det L_r^{\Pin^-}\\
\text{Line}&1&\String&\Tring&\det L_r^{\Tring}\\ \hline
\text{Plane}&2&\Spin&\Spin^c&\mathcal O_r\\
\text{Plane}&2&\String&\Stringc&\mathcal O_r
\end{array}
\]
With respect to the change-of-structure-group map $BH(r)\to BG(r)$, we have
\begin{equation}\label{eq:two-pattern-pullback}
L_r^{G}|_{BH(r)}\cong L_r^{H},\qquad D_r|_{BH(r)}\cong\mathbb R^q.
\end{equation}
Moreover, rank stabilization gives $L_{r+1}^{G}|_{BG(r)}\cong L_r^{G}\oplus\mathbb R$ and $D_{r+1}|_{BG(r)}\cong D_r$, with analogous isomorphisms on the $BH(r)$ side.

\begin{lemma}\label{lem:smith-classifying-spheres}
Let $r\geq1$. For each of the five families above, there are natural homotopy equivalences
\[
S(L_{r+1}^{G})\simeq BG(r),\qquad S(L_{r+1}^{H})\simeq BH(r),\qquad S(D_r)\simeq BH(r).
\]
These identifications are compatible with~\eqref{eq:two-pattern-pullback} and with the bundle isomorphisms under rank stabilization.
\end{lemma}

\begin{proof}
For $O$, $SO$, $\Spin$, $\Pin^-$, and $\Spin^c$, the first two identifications follow from the homogeneous spaces
\[
O(r+1)/O(r)\cong S^r,\qquad
SO(r+1)/SO(r)\cong S^r
\]
and the corresponding identifications for $\Spin$, $\Pin^-$, and $\Spin^c$. The third identification follows from
\[
O(r)/SO(r)\cong S^0,\qquad
\Pin^-(r)/\Spin(r)\cong S^0,\qquad
\Spin^c(r)/\Spin(r)\cong S^1.
\]

For the String refinements, the analogous identifications follow from the stable naturality of $\lambda$ and the definitions of the classifying spaces as homotopy pullbacks. All these identifications are natural by construction.
\end{proof}

\subsection{The line-bundle pattern}\label{subsec:smith-line-pattern}

In the first pattern, $q=1$ and $D_r$ is the real determinant line bundle. The corresponding pairs of structure groups are $(H,G)=(SO,O)$, $(\Spin,\Pin^-)$, and $(\String,\Tring)$.

\begin{thm}[Smith diagram for the line-bundle pattern]\label{thm:smith-line-unified}\label{thm:smith-oriented}\label{thm:smith-pin}\label{thm:smith-tring}
Let $(H,G)$ be one of the three pairs above, and let $n\geq1$. The following diagram is homotopy commutative, and each row and column is a Smith fiber sequence.
\[
\begin{tikzcd}[column sep=large,row sep=large]
BH(n)^{-(L_n^{H}\oplus\mathbb R^2)}\arrow[r]\arrow[d]
&BG(n)^{-(L_n^{G}\oplus D_n\oplus\mathbb R)}\arrow[r]\arrow[d]
&BG(n)^{-(L_n^{G}\oplus\mathbb R)}\arrow[d]\\
BH(n+1)^{-(L_{n+1}^{H}\oplus\mathbb R)}\arrow[r]\arrow[d]
&BG(n+1)^{-(L_{n+1}^{G}\oplus D_{n+1})}\arrow[r]\arrow[d]
&BG(n+1)^{-L_{n+1}^{G}}\arrow[d]\\
BH(n+1)^{-\mathbb R}\arrow[r]
&BG(n+1)^{-D_{n+1}}\arrow[r]
&BG(n+1)^{0}
\end{tikzcd}
\]

\end{thm}

\subsection{The oriented two-plane pattern}\label{subsec:smith-plane-pattern}

In the second pattern, $q=2$ and the complex determinant line bundle is regarded as an oriented real two-plane bundle $D_r$. The corresponding pairs of structure groups are $(H,G)=(\Spin,\Spin^c)$ and $(\String,\Stringc)$.

\begin{thm}[Smith diagram for the two-plane pattern]\label{thm:smith-plane-unified}\label{thm:smith-complex}\label{thm:smith-stringc}
Let $(H,G)$ be one of the two pairs above, and let $n\geq1$. The following diagram is homotopy commutative, and each row and column is a Smith fiber sequence.
\[
\begin{tikzcd}[column sep=large,row sep=large]
BH(n)^{-(L_n^{H}\oplus\mathbb R^3)}\arrow[r]\arrow[d]
&BG(n)^{-(L_n^{G}\oplus D_n\oplus\mathbb R)}\arrow[r]\arrow[d]
&BG(n)^{-(L_n^{G}\oplus\mathbb R)}\arrow[d]\\
BH(n+1)^{-(L_{n+1}^{H}\oplus\mathbb R^2)}\arrow[r]\arrow[d]
&BG(n+1)^{-(L_{n+1}^{G}\oplus D_{n+1})}\arrow[r]\arrow[d]
&BG(n+1)^{-L_{n+1}^{G}}\arrow[d]\\
BH(n+1)^{-\mathbb R^2}\arrow[r]
&BG(n+1)^{-D_{n+1}}\arrow[r]
&BG(n+1)^{0}
\end{tikzcd}
\]

\end{thm}

\begin{proof}[Proof of Theorems~\ref{thm:smith-line-unified} and~\ref{thm:smith-plane-unified}]
By Lemma~\ref{lem:smith-classifying-spheres}, the vertical sequences are the Smith fiber sequences associated with $L_{n+1}^{G}$ and $L_{n+1}^{H}$, while the horizontal sequences are those associated with $D_r$. Homotopy commutativity follows from naturality under pullbacks of bundles and compatibility of composition of Smith maps with direct sums of bundles.
\end{proof}

\section{Geometric bordism groups and their computation}\label{sec:geometric-groups}

Using the Smith diagrams from the preceding section, we derive short exact sequences describing the geometric bordism groups in the five families. We then reduce the computation of their kernels to the parity of Euler characteristics of closed manifolds with the respective tangential structures, proving the theorems in the introduction.

\subsection{Two geometric models}\label{subsec:two-geometric-models}

\begin{defi}[Bordism groups of the line-bundle pattern]\label{def:SO1}\label{def:Spin1}
Let $H\in\{SO,\Spin,\String\}$ and $n\geq1$. Consider bordism classes of pairs $(M,\eta)$ consisting of a closed $n$-manifold $M$ with an $H$-structure and a line subbundle $\eta\subset TM\oplus\mathbb R$, with addition given by disjoint union. A direct bordism consists of a compact $(n+1)$-manifold $W$ with an $H$-structure and a line subbundle $\zeta\subset TW$, whose boundary restrictions satisfy the normal convention in Definition~\ref{def:cobordism_group_A}. We use the equivalence relation generated by direct bordism as the bordism relation. For $H=SO,\Spin,\String$, we denote these groups by $\OrLine$, $\SpinLine$, and $\StringLine$, respectively.
\end{defi}

\begin{defi}[Bordism groups of the oriented two-plane pattern]\label{def:Spin2}
Let $H\in\{\Spin,\String\}$ and $n\geq1$. Consider bordism classes of pairs $(M,\eta)$ consisting of a closed $n$-manifold $M$ with an $H$-structure and an oriented rank-two subbundle $\eta\subset TM\oplus\mathbb R^2$, with addition given by disjoint union. A direct bordism consists of a compact $(n+1)$-manifold $W$ with an $H$-structure and an oriented rank-two subbundle $\zeta\subset TW\oplus\mathbb R$, whose boundary restrictions satisfy the same normal convention. We use the equivalence relation generated by direct bordism as the bordism relation. For $H=\Spin,\String$, we denote these groups by $\SpinPlane$ and $\StringPlane$, respectively.
\end{defi}

By Proposition~\ref{prop:direct-bordism}, direct bordism itself is an equivalence relation for $n\neq1$ in the line-bundle pattern and for all $n\geq1$ in the two-plane pattern.

\begin{prop}\label{prop:unified-geometric-pt}
Let $n\geq1$ and fix one of the five families in the preceding section.
Then the corresponding set of geometric bordism classes is an abelian group under disjoint union and is naturally isomorphic to either of the following groups:
\begin{equation}\label{eq:unified-geometry}
\Omega_n(BG(n);L_n^{G}\oplus D_n)\cong\pi_{-q}\bigl(BG(n)^{-(L_n^{G}\oplus D_n)}\bigr).
\end{equation}

\end{prop}

\begin{proof}
First consider the line-bundle pattern. Given a line subbundle $\eta\subset TM\oplus\mathbb R$, put $E=\eta^\perp$. The $H$-structure equips $TM\oplus\mathbb R=E\oplus\eta$ with an orientation or a refinement thereof, so that $\eta\cong\det E$. If $H=\Spin$, then $E$ has a Pin$^-$ structure, and if $H=\String$, then $E$ has a $\Tring$ structure. Conversely, specifying the corresponding $G$-structure on $E$ and an isomorphism $TM\oplus\mathbb R\cong E\oplus\det E$ recovers the $H$-structure on $TM$ and the line subbundle of $TM\oplus\mathbb R$.

In the two-plane pattern, given an oriented two-plane subbundle $\eta\subset TM\oplus\mathbb R^2$, again put $E=\eta^\perp$. The Spin structure on $TM\oplus\mathbb R^2=E\oplus\eta$ induces a Spin$^c$ structure on $E$ with determinant complex line bundle $\eta$. If $TM$ has a String structure, the resulting structure on $E$ is a $\Stringc$ structure. The converse construction is analogous. This gives the first isomorphism, which preserves disjoint union, and the second isomorphism and the group structure follow from Theorem~\ref{thm:unstable-pt}.
\end{proof}

\subsection{The common exact sequence}\label{subsec:common-smith-exact}

We now describe the computational mechanism common to the five families. Let $n\geq1$, fix one of the five triples $(H,G,q)$ from the preceding section, and retain the notation $L_r^G,L_r^H,D_r$. Let $\Omega_j^G$ and $\Omega_j^H$ denote the bordism groups of the corresponding stable tangential structures, taken to be zero for $j<0$. Our aim is to compute
\[
\pi_{-q}\bigl(BG(n)^{-(L_n^{G}\oplus D_n)}\bigr),
\]
which is identified with the geometric bordism group by the unstable Pontryagin--Thom theorem.

\begin{lemma}\label{lem:smith-connectivity}
Let $Z$ be a connected space and $V$ a virtual real vector bundle over $Z$. If $\operatorname{rank}V>0$, then $\pi_0(Z^V)=0$, and if $\operatorname{rank}V=0$, then
\[
\pi_0(Z^V)\cong
\begin{cases}
\mathbb Z,&V\text{ is orientable},\\
\mathbb Z/2,&V\text{ is nonorientable}
\end{cases}.
\]
Moreover, for the five families in the preceding section, if $r\geq1$ and $s>0$, then
\[
\pi_{-s}(BG(r)^{-L_r^{G}})\cong\Omega_{r-s}^G,\qquad
\pi_{-s}(BH(r)^{-L_r^{H}})\cong\Omega_{r-s}^H.
\]

\end{lemma}

\begin{proof}
The first assertion follows from connectivity of the Thom spectrum, and the rank-zero assertion follows from the Hurewicz theorem and the computation of $H_0$ with local coefficients. The last assertion follows from the Smith fiber sequence for rank stabilization
\[
BG(r)^{-(L_r^{G}\oplus\mathbb R)}\longrightarrow BG(r+1)^{-L_{r+1}^{G}}\longrightarrow BG(r+1)^0.
\]
Since the rightmost term is connective, for $s>0$ we have
\[
\pi_{-s}(BG(r)^{-L_r^{G}})\cong\pi_{-(s+1)}(BG(r+1)^{-L_{r+1}^{G}}).
\]
After iterating rank stabilization, the stable Pontryagin--Thom theorem identifies this group with $\Omega_{r-s}^G$. The argument for $BH(r)$ is analogous.
\end{proof}

\begin{remark}
Ebert observed that $\pi_0(BO(r)^{-L_r^O})$ and $\pi_0(BSO(r)^{-L_r^{SO}})$ are isomorphic to the bordism groups defined by Reinhart \cite{Ebert13}. Computations of the low-degree homotopy groups of $BSO(r)^{-L_r^{SO}}$ using the Adams spectral sequence are also known \cite{BDS15}.
\end{remark}

First apply $\pi_{-(q+1)}$ to the middle horizontal Smith fiber sequence of rank $n+1$ in the preceding section. Using Lemma~\ref{lem:smith-connectivity} and suspensions by trivial bundles, we obtain the following part of the exact sequence.
\[
\Omega_n^H\xrightarrow{\,\iota\,}
\pi_{-(q+1)}\bigl(BG(n+1)^{-(L_{n+1}^{G}\oplus D_{n+1})}\bigr)
\xrightarrow{\,z\,}\Omega_{n-q}^G
\longrightarrow\Omega_{n-1}^H.
\]
Here $\iota$ assigns a trivial rank-$q$ subbundle to a manifold with an $H$-structure. There is also a map $u$ forgetting the specified subbundle in the geometric data representing the middle group, and $u\circ\iota=\mathrm{id}$. The same forgetful map gives a left inverse one degree lower, so the connecting homomorphism $\Omega_{n-q}^G\to\Omega_{n-1}^H$ in the exact sequence is zero and $z$ is surjective. Thus we obtain a natural group isomorphism
\begin{equation}\label{eq:common-B-splitting}
(z,u):\pi_{-(q+1)}\bigl(BG(n+1)^{-(L_{n+1}^{G}\oplus D_{n+1})}\bigr)
\xrightarrow{\ \cong\ }\Omega_{n-q}^G\oplus\Omega_n^H.
\end{equation}
Here $z$ is the Smith homomorphism taking the zero locus of a transverse section of the determinant bundle, and $u$ forgets the subbundle.

Next we use the middle vertical fiber sequence
\[
BG(n)^{-(L_n^{G}\oplus D_n\oplus\mathbb R)}
\longrightarrow BG(n+1)^{-(L_{n+1}^{G}\oplus D_{n+1})}
\longrightarrow BG(n+1)^{-D_{n+1}}.
\]
For the rightmost term, we have
\[
\pi_{-q}\bigl(BG(n+1)^{-D_{n+1}}\bigr)\cong
\begin{cases}
\mathbb Z/2,&q=1,\\
\mathbb Z,&q=2
\end{cases},\qquad
\pi_{-(q+1)}\bigl(BG(n+1)^{-D_{n+1}}\bigr)=0.
\]
Indeed, for $q=1$, $D_{n+1}$ is a nonorientable line bundle, and for $q=2$, it is an oriented two-plane bundle, so the assertions follow from Lemma~\ref{lem:smith-connectivity}. Hence the long exact sequence and~\eqref{eq:common-B-splitting} give a natural surjection
\begin{equation}\label{eq:common-epsilon}
\epsilon:\pi_{-q}\bigl(BG(n)^{-(L_n^{G}\oplus D_n)}\bigr)
\longrightarrow\Omega_{n-q}^G\oplus\Omega_n^H.
\end{equation}
Its geometric interpretation is
\begin{equation}\label{eq:geometric-epsilon-two-patterns}
\epsilon[M,\eta]=\bigl([Z(s)],[M]\bigr),
\end{equation}
where $s$ is a section of the specified rank-$q$ subbundle $\eta$ transverse to the zero section. For this geometric interpretation of the Smith homomorphism, see \cite[Proposition~3.17]{Detal26}. Below we denote the image of the middle vertical map
\begin{equation}\label{eq:common-p-map}
p:\pi_{-q}\bigl(BG(n+1)^{-(L_{n+1}^{G}\oplus D_{n+1})}\bigr)
\longrightarrow
\begin{cases}
\mathbb Z/2,&q=1,\\
\mathbb Z,&q=2
\end{cases}
\end{equation}
by $\operatorname{im}p$. The above long exact sequence gives the short exact sequence
\begin{equation}\label{eq:common-exact-sequence}
0\longrightarrow\frac{\pi_{-q}(BG(n+1)^{-D_{n+1}})}{\operatorname{im}p}
\longrightarrow\pi_{-q}\bigl(BG(n)^{-(L_n^{G}\oplus D_n)}\bigr)
\xrightarrow{\ \epsilon\ }\Omega_{n-q}^G\oplus\Omega_n^H\longrightarrow0.
\end{equation}

To compute the image of $p$, we also use the fact that the zero locus and the forgetful map in one higher dimension realize arbitrary pairs of bordism classes. More precisely, considering the same horizontal long exact sequence one degree higher gives
the surjection
\begin{equation}\label{eq:upper-splitting}
\pi_{-q}\bigl(BG(n+1)^{-(L_{n+1}^{G}\oplus D_{n+1})}\bigr)
\longrightarrow\Omega_{n+1-q}^G\oplus\Omega_{n+1}^H,
\qquad [M,\eta]\longmapsto([Z(s)], [M]).
\end{equation}

We now compute the image of $p$ and determine the kernel in each family in terms of the parity of Euler characteristics.

\subsection{The line-bundle pattern}\label{subsec:line-bordism-calculation}

First consider the line-bundle pattern $q=1$. The corresponding three families are $(SO,O)$, $(\Spin,\Pin^-)$, and $(\String,\Tring)$. The part of the long exact sequence associated with the middle vertical Smith fiber sequence of the preceding section that we need is
\[
\begin{tikzcd}[column sep=small,ampersand replacement=\&]
\pi_{-1}\bigl(BG(n+1)^{-(L_{n+1}^{G}\oplus D_{n+1})}\bigr)
\arrow[r,"p"]
\&\mathbb Z/2\arrow[r,"\partial"]
\&\pi_{-1}\bigl(BG(n)^{-(L_n^{G}\oplus D_n)}\bigr)
\arrow[r,"\epsilon"]
\&\Omega_{n-1}^G\oplus\Omega_n^H\arrow[r]
\&0
\end{tikzcd},
\]
where $\partial$ is the connecting homomorphism and the surjectivity of $\epsilon$ was established in the preceding subsection. Thus~\eqref{eq:common-exact-sequence} gives
\begin{equation}\label{eq:line-short-exact}
0\longrightarrow (\mathbb Z/2)/\operatorname{im}p
\longrightarrow \pi_{-1}\bigl(BG(n)^{-(L_n^{G}\oplus D_n)}\bigr)
\xrightarrow{\ \epsilon\ }\Omega_{n-1}^G\oplus\Omega_n^H\longrightarrow0.
\end{equation}
Computing the kernel therefore reduces to determining whether $p$ is zero or surjective.

Represent an element of the domain in~\eqref{eq:common-p-map} by a closed $(n+1)$-manifold $M$ with an $H$-structure and a line subbundle $\eta\subset TM\oplus\mathbb R$. Putting $E=\eta^\perp$, the map $p$ counts the zeros of a generic section of $E$ modulo two. Thus
\[
p[M,\eta]=\langle w_{n+1}(E),[M]\rangle\in\mathbb Z/2.
\]
The Whitney product formula and $TM\oplus\mathbb R\cong E\oplus\eta$ give
\[
w_{n+1}(E)=w_{n+1}(TM)+w_n(E)w_1(\eta).
\]
The evaluation of the first term is $\chi(M)\bmod2$. If $Z(s)$ is the zero locus of a transverse section $s$ of $\eta$, then the Poincar\'e dual of $[Z(s)]$ is $w_1(\eta)$. Since $E|_{Z(s)}$ and $TZ(s)\oplus\mathbb R$ are stably isomorphic over $Z(s)$, the evaluation of the second term equals $\chi(Z(s))\bmod2$. It follows that
\begin{equation}\label{eq:line-euler-parity}
p[M,\eta]=\chi(Z(s))+\chi(M)\pmod2.
\end{equation}

By the surjectivity of~\eqref{eq:upper-splitting}, the pair $([Z(s)],[M])$ can realize any element of $\Omega_n^G\oplus\Omega_{n+1}^H$. Since the reduction modulo two of the Euler characteristic is a bordism invariant for each tangential structure,
\begin{equation}\label{eq:line-image-p}
\operatorname{im}p
=\operatorname{im}\bigl(\chi\bmod2:\Omega_n^G\to\mathbb Z/2\bigr)
+\operatorname{im}\bigl(\chi\bmod2:\Omega_{n+1}^H\to\mathbb Z/2\bigr).
\end{equation}
Thus the kernel is $\mathbb Z/2$ if both Euler characteristics are always even, and zero otherwise.

\begin{thm}\label{thm:line-unified-kernel}
Let $(H,G)=(SO,O)$, $(\Spin,\Pin^-)$, or $(\String,\Tring)$, and let $n\geq1$. Then the homomorphism $\epsilon$ in~\eqref{eq:common-epsilon} is surjective, and
\[
\ker\epsilon\cong
\begin{cases}
\mathbb Z/2,&\substack{\text{all closed $G$ $n$-manifolds and closed $H$ $(n+1)$-manifolds}\\\text{have even Euler characteristic}},\\
0,&\text{otherwise}
\end{cases}.
\]

\end{thm}

Apply this formula to $(H,G)=(SO,O)$. The Euler characteristic of a closed odd-dimensional manifold is always zero, by Poincar\'e duality and the orientation double cover. In even dimensions, $\chi(\mathbb RP^{2a})=1$, so odd Euler characteristics can occur without an orientability assumption. For a closed oriented $(4a+2)$-manifold, the middle-dimensional intersection form is nondegenerate and alternating by Poincar\'e duality, so the middle Betti number, and hence the Euler characteristic, is even. In contrast, $\mathbb CP^{2a}$ has odd Euler characteristic in dimension $4a$. Thus the general formula for the line-bundle pattern gives
\[
\ker\bigl(\epsilon:\OrLine\to\Omega_{n-1}^O\oplus\Omega_n^{SO}\bigr)
\cong\begin{cases}\mathbb Z/2,&n\equiv1\pmod4,\\0,&n\not\equiv1\pmod4\end{cases},
\]
proving Theorem~\ref{thm:oriented-line}.

For $(H,G)=(\Spin,\Pin^-)$, we use the following result of Hoekzema--Stehouwer--Vesel\'a \cite{HSV25,HSV26}. We include a proof because of its relevance to the subsequent discussion of the String refinements.

\begin{lemma}[Euler characteristics of Pin$^-$ manifolds]\label{lem:pin-euler-parity}
The Euler characteristic of a closed Pin$^-$ $d$-manifold is even if $d$ is odd or $d\equiv4\pmod8$. On the other hand, examples with odd Euler characteristic exist for $d\equiv0,2,6\pmod8$.
\end{lemma}

\begin{proof}
In odd dimensions, $\chi=0$ by Poincar\'e duality. The Pin$^-$ condition is equivalent to the vanishing of the second Wu class $v_2=w_2+w_1^2$. For $d=4$, we have $w_4=v_2^2=0$. For $d=8a+4$ ($a\geq1$), we use the Adem relation
\[
\operatorname{Sq}^{4a+2}=\operatorname{Sq}^2\operatorname{Sq}^{4a}+\operatorname{Sq}^2\operatorname{Sq}^{4a-1}\operatorname{Sq}^1.
\]
For any $y\in H^{4a+2}(M;\mathbb Z/2)$, the operation $\operatorname{Sq}^2$ applied last in each term on the right acts on $H^{d-2}(M;\mathbb Z/2)$. By the Wu formula and $v_2=0$, we have
\[
\langle v_{4a+2}\cup y,[M]\rangle=\langle\operatorname{Sq}^{4a+2}y,[M]\rangle=0.
\]
Poincar\'e duality implies $v_{4a+2}=0$, and hence $w_d=v_{d/2}^2=0$, which means that the Euler characteristic is even. Examples with odd Euler characteristic for $d\equiv0,2,6\pmod8$ are given by $\mathbb HP^{2a}$, $\mathbb RP^{8a+2}$, and $\mathbb RP^{8a+6}$, respectively. For the real projective spaces, this follows from $w(T\mathbb RP^d)=(1+x)^{d+1}$, where $x\in H^1(\mathbb RP^d;\mathbb Z/2)$ is the standard generator, which gives $w_2+w_1^2=0$ for $d\equiv2\pmod4$.
\end{proof}

A closed Spin $d$-manifold with odd Euler characteristic exists if and only if $d\equiv0\pmod8$. In odd dimensions and for $d\equiv2\pmod4$, we use Poincar\'e duality, and for $d\equiv4\pmod8$, we use $\chi\equiv\operatorname{sign}\pmod2$ and the divisibility by two of the signature of Spin manifolds \cite{Ochanine81}. For $d=8a$, $\mathbb HP^{2a}$ gives an example with odd Euler characteristic. Thus
\[
\ker\bigl(\epsilon:\SpinLine\to\Omega_{n-1}^{\Pin^-}\oplus\Omega_n^{\Spin}\bigr)
\cong\begin{cases}
\mathbb Z/2,&n\equiv1,3,4,5\pmod8,\\
0,&n\equiv0,2,6,7\pmod8
\end{cases},
\]
giving the computation of the kernel in Theorem~\ref{thm:spin-line}.

Finally, consider $(H,G)=(\String,\Tring)$. The short exact sequence for the line-bundle pattern also holds in this case, giving
\begin{equation}\label{eq:string-line-unified}
0\longrightarrow(\mathbb Z/2)/\operatorname{im}p
\longrightarrow\StringLine\xrightarrow{\ \epsilon\ }
\Omega_{n-1}^{\Tring}\oplus\Omega_n^{\String}\longrightarrow0.
\end{equation}
The kernel is determined by the parity of the Euler characteristics of closed $\Tring$ $n$-manifolds and closed String $(n+1)$-manifolds. For the latter, by Hoekzema \cite{Hoekzema20}, odd Euler characteristics occur only in dimensions divisible by $16$, and are realized by products of the String manifold $\mathbb OP^2$. We therefore determine the necessary and sufficient condition on the $\Tring$ side.

\begin{lemma}[Formal Wu classes of String bundles]\label{lem:string-formal-wu}
Let $F$ be a real vector bundle with $w_i(F)=0$ for $1\leq i\leq7$. Define its formal Wu class $v(F)$ by $\operatorname{Sq}(v(F))=w(F)$. Then
\[
v_j(F)=0\qquad(8\nmid j).
\]
In particular, this holds for String bundles.
\end{lemma}

\begin{proof}
Let $U$ be the Thom class of $F$ and $\chi$ the conjugation anti-automorphism of the Steenrod algebra. By the Thom identity,
\[
\chi(\operatorname{Sq})(U)=v(F)^{-1}U.
\]
By the assumption and the recursive formula for conjugation, $\chi(\operatorname{Sq}^i)(U)=0$ for $1\leq i\leq7$. Applying the conjugation anti-automorphism to Hoekzema's decomposition \cite[Proposition~5.1]{Hoekzema20} gives
\[
\chi(\operatorname{Sq}^j)(U)=0\qquad(8\nmid j).
\]
Hence $v(F)^{-1}$, and consequently $v(F)$, has components only in degrees divisible by $8$. Note that formal Wu classes are multiplicative under direct sums by the Cartan formula and agree with the usual Wu classes when $F=TM$.
\end{proof}

\begin{thm}[Euler characteristics of Tring$^-$ manifolds]\label{thm:tring-euler-complete}
Let $d\geq1$. There exists a closed $d$-dimensional $\Tring$ manifold with odd Euler characteristic if and only if
\[
d\equiv0,2,6,14\pmod{16}.
\]

\end{thm}

\begin{proof}
Put $a=w_1(TM)$ and $L=\det(TM)$. The bundle $TM\oplus L$ is String, and multiplicativity of formal Wu classes gives
\[
v(TM)=v(TM\oplus L)v(\det(TM))^{-1},\qquad
v(L)^{-1}=1+a+a^3+a^7+a^{15}+\cdots.
\]
The latter formula follows from $v(L)=1+\sum_{k\geq0}a^{2^k}$. By Lemma~\ref{lem:string-formal-wu}, we obtain
\[
v_j(TM)=0\qquad(j\equiv2,4,5,6\pmod8).
\]
Since $w_d(TM)=v_j(TM)^2$ for $d=2j$, the Euler characteristic is even for $d\equiv4,8,10,12\pmod{16}$.

For sufficiency, let $M=\mathbb RP^{8k+6}$ ($k\geq0$), and let $\gamma$ be the tautological real line bundle. Then
\[
(TM\oplus\det(TM))\oplus\mathbb R\cong(8k+8)\gamma.
\]
Since $(1+a)^{8k+8}=(1+a^8)^{k+1}$, this bundle satisfies $w_1=w_2=w_4=0$. Moreover, injectivity of the coefficient reduction $H^4(M;\mathbb Z)\to H^4(M;\mathbb Z/2)$ implies $\lambda=0$. Thus $M$ admits a $\Tring$ structure. The manifold $\mathbb RP^2$ also admits a Pin$^-$ structure, and since $H^4(\mathbb RP^2;\mathbb Z)=0$, it admits a $\Tring$ structure. Hence the following manifolds give the desired examples.
\[
\begin{array}{c|c}
d\bmod16 & M\\ \hline
0 & (\mathbb OP^2)^k\quad(k\geq1)\\
2 & \mathbb RP^2\times(\mathbb OP^2)^k\\
6 & \mathbb RP^{16k+6}\\
14 & \mathbb RP^{16k+14}
\end{array}
\]
\end{proof}

Applying Theorem~\ref{thm:tring-euler-complete} and Hoekzema's result to Theorem~\ref{thm:line-unified-kernel}, we see that $\operatorname{im}p=\mathbb Z/2$ if and only if $n\equiv0,2,6,14,15\pmod{16}$. Therefore
\[
\ker\bigl(\epsilon:\StringLine\to\Omega_{n-1}^{\Tring}\oplus\Omega_n^{\String}\bigr)
\cong\begin{cases}
0,&n\equiv0,2,6,14,15\pmod{16},\\
\mathbb Z/2,&\text{otherwise}
\end{cases},
\]
giving the computation of the kernel in Theorem~\ref{thm:string-line}.

\subsection{The oriented two-plane pattern}\label{subsec:plane-bordism-calculation}

Next consider the two-plane pattern $q=2$. The corresponding families are $(\Spin,\Spin^c)$ and $(\String,\Stringc)$. Taking the homotopy groups in degree $-2$ of the same vertical Smith fiber sequence, we obtain
\[
\begin{tikzcd}[column sep=small,ampersand replacement=\&]
\pi_{-2}\bigl(BG(n+1)^{-(L_{n+1}^{G}\oplus D_{n+1})}\bigr)
\arrow[r,"p"]
\&\mathbb Z\arrow[r,"\partial"]
\&\pi_{-2}\bigl(BG(n)^{-(L_n^{G}\oplus D_n)}\bigr)
\arrow[r,"\epsilon"]
\&\Omega_{n-2}^G\oplus\Omega_n^H\arrow[r]
\&0
\end{tikzcd}.
\]
Here again $\partial$ is the connecting homomorphism and $\epsilon$ is surjective. Thus the corresponding short exact sequence is
\begin{equation}\label{eq:plane-short-exact}
0\longrightarrow\mathbb Z/\operatorname{im}p
\longrightarrow \pi_{-2}\bigl(BG(n)^{-(L_n^{G}\oplus D_n)}\bigr)
\xrightarrow{\ \epsilon\ }\Omega_{n-2}^G\oplus\Omega_n^H\longrightarrow0.
\end{equation}
Here $p$ takes the Euler number of the complementary bundle. Indeed, represent an element of the domain in~\eqref{eq:common-p-map} by a closed $(n+1)$-manifold $M$ with an $H$-structure and an oriented two-plane subbundle $\eta\subset TM\oplus\mathbb R^2$, and put $E=\eta^\perp$. Then
\begin{equation}\label{eq:plane-euler-number}
p[M,\eta]=\langle e(E),[M]\rangle\in\mathbb Z,
\end{equation}
where $E$ is oriented by the orientations of $M$ and $\eta$.

If $n$ is even, $\operatorname{rank}E=n+1$ is odd. The Euler class of an oriented real vector bundle of odd rank is $2$-torsion, so its evaluation on the fundamental class of the closed oriented manifold $M$ is zero. Hence $\operatorname{im}p=0$ and $\ker\epsilon\cong\mathbb Z$ in this case.

If $n$ is odd, we may take $M=S^{n+1}$ and choose $\eta$ to be the trivial two-plane bundle. Then $E=TS^{n+1}$, and the required structures exist in both the Spin and String cases by stable parallelizability of spheres. Thus
\[
p[S^{n+1},\mathbb R^2]=\chi(S^{n+1})=2.
\]
Hence $2\mathbb Z\subset\operatorname{im}p$, and the image is either $2\mathbb Z$ or $\mathbb Z$. The class $w_{n+1}(E)$ is the reduction modulo two of $e(E)$, and by the Whitney product formula and $w_1(\eta)=0$, we obtain
\[
w_{n+1}(E)=w_{n+1}(TM)+w_{n-1}(E)w_2(\eta).
\]
If $Z(s)$ is the zero locus of a transverse section $s$ of $\eta$, then, as in the line-bundle case, the first term corresponds to $\chi(M)\bmod2$ and the second to $\chi(Z(s))\bmod2$. Thus
\begin{equation}\label{eq:plane-euler-parity}
p[M,\eta]\equiv\chi(Z(s))+\chi(M)\pmod2.
\end{equation}
Moreover, by~\eqref{eq:upper-splitting}, the pair $([Z(s)],[M])$ can realize any element of $\Omega_{n-1}^G\oplus\Omega_{n+1}^H$.

\begin{thm}\label{thm:plane-unified-kernel}
Let $(H,G)=(\Spin,\Spin^c)$ or $(\String,\Stringc)$, and let $n\geq1$. Then the homomorphism $\epsilon$ in~\eqref{eq:common-epsilon} is surjective, and
\[
\ker\epsilon\cong
\begin{cases}
\mathbb Z,&n\text{ is even},\\
\mathbb Z/2,&\substack{n\text{ is odd and all closed $G$ $(n-1)$-manifolds}\\\text{and closed $H$ $(n+1)$-manifolds have even Euler characteristic}},\\
0,&\text{otherwise}
\end{cases}.
\]

\end{thm}

For $(H,G)=(\Spin,\Spin^c)$, a closed Spin$^c$ manifold can have odd Euler characteristic only in dimensions divisible by $4$. Indeed, Spin$^c$ manifolds are oriented, and in dimensions divisible by $4$, $\mathbb CP^{2a}$ gives examples with odd Euler characteristic. In contrast, for Spin manifolds such dimensions must be divisible by $8$, as discussed above. Therefore
\[
\ker\bigl(\epsilon:\SpinPlane\to\Omega_{n-2}^{\Spin^c}\oplus\Omega_n^{\Spin}\bigr)
\cong\begin{cases}
\mathbb Z,&n\equiv0,2,4,6\pmod8,\\
\mathbb Z/2,&n\equiv3\pmod8,\\
0,&n\equiv1,5,7\pmod8
\end{cases},
\]
giving the computation of the kernel in Theorem~\ref{thm:spin-plane}.

The same short exact sequence holds for $(H,G)=(\String,\Stringc)$, giving
\begin{equation}\label{eq:string-plane-unified}
0\longrightarrow\mathbb Z/\operatorname{im}p
\longrightarrow\StringPlane\xrightarrow{\ \epsilon\ }
\Omega_{n-2}^{\Stringc}\oplus\Omega_n^{\String}\longrightarrow0.
\end{equation}
If $n$ is even, then $\ker\epsilon\cong\mathbb Z$. To compute the kernel for odd $n$, we examine the parity of Euler characteristics of closed $\Stringc$ manifolds.

A $\Stringc$ structure in our convention consists of a String structure on $F=TM\oplus L_{\mathbb R}$, where $L$ is the determinant complex line bundle of the Spin$^c$ manifold $M$. In particular, $\lambda(F)=0$, which implies $p_1(TM)+c^2=0$ for $c=c_1(L)$.

\begin{thm}[Euler characteristics of String$^c$ manifolds]\label{thm:stringc-euler-complete}
Let $d\geq1$. There exists a closed $d$-dimensional $\Stringc$ manifold with odd Euler characteristic if and only if
\[
d\equiv0,12\pmod{16},\qquad\text{or}\qquad d\equiv4\pmod{16}\text{ and }d\geq20.
\]

\end{thm}

\begin{proof}
We first establish the restrictions on parity. Since $M$ is oriented, $\chi(M)=0$ in odd dimensions, and $\chi(M)$ is even for $d\equiv2\pmod4$ because the middle-dimensional intersection form is alternating. Put $x=w_2(L_{\mathbb R})=c_1(L)\bmod2$ and $F=TM\oplus L_{\mathbb R}$. Since $F$ is String, Lemma~\ref{lem:string-formal-wu} applies. Moreover, $w(L_{\mathbb R})=1+x$ and $\operatorname{Sq}(x)=x+x^2$ give
\[
v(L_{\mathbb R})^{-1}=1+x+x^3+x^7+x^{15}+\cdots.
\]
Hence the only degrees in which $v(TM)=v(F)v(L_{\mathbb R})^{-1}$ can have nonzero components are $0,2,6$ modulo $8$. For $d=2j$, we have $w_d(TM)=v_j(TM)^2$, so $\chi(M)$ is even for $d\equiv2,6,8,10,14\pmod{16}$. It remains to exclude $d=4$. In this case, $\lambda(F)=0$ implies $p_1(TM)=-c^2$, and the Hirzebruch signature theorem and integrality of the Spin$^c$ Dirac index give
\[
\langle c^2,[M]\rangle=-3\operatorname{sign}(M),\qquad
\operatorname{ind}D^c=\frac{\langle c^2,[M]\rangle-\operatorname{sign}(M)}8=-\frac{\operatorname{sign}(M)}2\in\mathbb Z.
\]
Thus $\operatorname{sign}(M)$ is even, and so is $\chi(M)$ by $\chi(M)\equiv\operatorname{sign}(M)\pmod2$.

We next construct $\Stringc$ manifolds with odd Euler characteristic. The $16$-dimensional Cayley plane $\mathbb OP^2$ is a String manifold with $\chi=3$, so its products give examples for $d\equiv0\pmod{16}$. For $d\equiv12\pmod{16}$, we use a smooth cubic hypersurface $X$ in $\mathbb CP^7$. The complex dimension of $X$ is $6$, and writing $h=c_1(\mathcal O_X(1))$, we have
\[
c(TX)=\frac{(1+h)^8}{1+3h},\qquad
c_1(TX)=5h,\qquad c_2(TX)=13h^2.
\]
Hence
\[
p_1(TX)=c_1(TX)^2-2c_2(TX)=-h^2.
\]

The line bundle $L=\mathcal O_X(1)$ satisfies $c_1(L)\equiv w_2(TX)\pmod2$, so $TX\oplus L_{\mathbb R}$ is Spin. Furthermore, $p_1(TX\oplus L_{\mathbb R})=0$, and by the Lefschetz hyperplane theorem, $H^4(X;\mathbb Z)\cong\mathbb Z$ has no $2$-torsion. Thus $\lambda(TX\oplus L_{\mathbb R})=0$, and $X$ admits a $\Stringc$ structure. Its Euler characteristic is
\[
\chi(X)=\langle c_6(TX),[X]\rangle=3[h^6]\frac{(1+h)^8}{1+3h}=93.
\]
The products $X\times(\mathbb OP^2)^k$ give examples for $d\equiv12\pmod{16}$.

Finally, we construct an example in dimension $d=20$. We adapt the construction of Cayley plane bundles in String bordism \cite{McTague14,HanHuang} to the $\Stringc$ setting. Let $h\in H^2(\mathbb CP^2;\mathbb Z)$ be the standard generator. Since the fourth Postnikov section of $B\Spin(9)$ is identified with $K(\mathbb Z,4)$, we can choose a principal $\Spin(9)$-bundle $P\to\mathbb CP^2$ satisfying $\lambda(P)=-h^2$. Using the action of $\Spin(9)$ on $\mathbb OP^2=F_4/\Spin(9)$, put
\[
Q=P\times_{\Spin(9)}\mathbb OP^2,\qquad
p:Q\to\mathbb CP^2.
\]
Since this action fixes the basepoint, $p$ admits a section $s$. The pullback of the vertical tangent bundle $T_vQ$ along $s$ is the bundle associated with the real spin representation $\Delta_9$. Taking half the sum of the squares of the weights $\frac12(\pm t_1\pm t_2\pm t_3\pm t_4)$ of $\Delta_9$ gives
\[
p_1(\Delta_9)=2p_1(\mathbb R^9)=4\lambda
\]
(cf.~\cite{HanHuang}). Since $\mathbb OP^2$ is $7$-connected, $p^*$ induces an isomorphism in cohomology in degrees at most $4$, with inverse $s^*$. Thus $T_vQ$ is Spin, and $p_1(T_vQ)=-4p^*h^2$, so
\[
p_1(TQ)=p_1(T_vQ)+p^*p_1(T\mathbb CP^2)
=(-4+3)p^*h^2=-p^*h^2.
\]
Taking $L=p^*\mathcal O_{\mathbb CP^2}(1)$, we see that $TQ\oplus L_{\mathbb R}$ is Spin and satisfies $p_1=0$. Since $H^4(Q;\mathbb Z)\cong\mathbb Z$ has no $2$-torsion, $\lambda=0$ follows, and $Q$ is a $\Stringc$ manifold. Furthermore, $\chi(Q)=3\cdot3=9$. Thus $Q\times(\mathbb OP^2)^k$ gives examples in dimensions $20+16k$.
\end{proof}

Applying Theorem~\ref{thm:stringc-euler-complete} and Hoekzema's result to Theorem~\ref{thm:plane-unified-kernel}, together with the preceding observations, gives
\[
\ker\bigl(\epsilon:\StringPlane\to\Omega_{n-2}^{\Stringc}\oplus\Omega_n^{\String}\bigr)
\cong\begin{cases}
\mathbb Z,&n\text{ even},\\
\mathbb Z/2,&n\equiv3,7,9,11\pmod{16}\text{ or }n=5,\\
0,&\text{otherwise}.
\end{cases},
\]
completing the computation of the kernel in Theorem~\ref{thm:string-plane}.

\section{Applications}\label{sec:applications}

In this section, we apply the computations of the bordism groups in Section~\ref{sec:geometric-groups} to geometric existence problems. First, using the three line-bundle families, we determine the dimensions in which a sphere bounds a manifold carrying a normal line field along its boundary. We then use $\SpinPlane$ to study the bordism relation of Bais--May Custodio--Torres.

\subsection{Manifolds with spherical boundary and a normal line field}

Let $H\in\{SO,\Spin,\String\}$. We seek a compact $(n+1)$-manifold $W$ with an $H$-structure and a line subbundle $\zeta\subset TW$ that agrees with the normal line bundle along the boundary $\partial W\cong S^n$. We equip $S^n$ with the $H$-structure induced from the standard disk. This is a question about the vanishing of the corresponding class in $\OrLine$, $\SpinLine$, or $\StringLine$.

\begin{thm}[Manifolds with spherical boundary and a normal line field]\label{thm:boundary-line}
Let $n\geq1$. The necessary and sufficient conditions for the existence of a compact $(n+1)$-manifold $W$ with an $H$-structure and boundary $S^n$, carrying a line subbundle $\zeta\subset TW$ that agrees with the normal direction along the boundary, are as follows.
\[
\begin{array}{c|c}
H&\text{Necessary and sufficient condition}\\\hline
SO&n\not\equiv1\pmod4\\
\Spin&n\equiv0,2,6,7\pmod8\\
\String&n\equiv0,2,6,14,15\pmod{16}
\end{array}
\]
\end{thm}

\begin{proof}
Let $n\geq1$ and consider the standard boundary data on $S^n$,
\[
[S^n,\mathbb R\subset TS^n\oplus\mathbb R],
\]
which defines an element of the corresponding line-bundle bordism group. Recall from Section~\ref{sec:geometric-groups} the following part of the Smith long exact sequence.
\[
\begin{tikzcd}[column sep=large,ampersand replacement=\&]
\mathbb Z/2\arrow[r,"\partial"]
\&\pi_{-1}\bigl(BG(n)^{-(L_n^{G}\oplus D_n)}\bigr)
\arrow[r,"\epsilon"]
\&\Omega_{n-1}^G\oplus\Omega_n^H
\end{tikzcd}
\]
Here $\partial(1)$ is the bordism class of the standard sphere $[S^n,\mathbb R\subset TS^n\oplus\mathbb R]$. This is stated for $SO$ in \cite[Proposition~A.5]{Ebert13}, and follows similarly for $\Spin$ and $\String$.

By the short exact sequences in Section~\ref{sec:geometric-groups}, this sphere class vanishes if and only if the kernel of the line-bundle homomorphism $\epsilon$ is zero. On the other hand, for $n\ge2$, direct bordism is itself an equivalence relation by Proposition~\ref{prop:direct-bordism}, so vanishing of the sphere class is equivalent to the existence of a single bordism $W$ as required. For $n=1$, the kernel is $\mathbb Z/2$ in all three families, and the sphere class does not vanish. Thus the exceptional case does not affect the existence criterion.

Substituting the computations of the kernels in Theorems~\ref{thm:oriented-line}, \ref{thm:spin-line}, and~\ref{thm:string-line} for $SO$, $\Spin$, and $\String$, respectively, gives the congruences in the table for $n\geq1$.
\end{proof}

\begin{remark}\label{rem:explicit-boundary-line}
For example, for $n\equiv2,6\pmod8$, removing an open disk from the Spin manifold $\mathbb RP^{n+1}$ gives a desired example $W$. Indeed,
\[
W\cong(S^n\times[-1,1])/((x,t)\sim(-x,-t)),
\]
and the involution reverses the sign of $\partial/\partial t$ but preserves the line field it spans. Hence this defines a line field on the quotient agreeing with the normal line bundle along the boundary. Moreover, for $n=8k-1$ ($k\geq1$),
\[
X_k=\mathbb{HP}^{2k}\#\underbrace{(S^1\times S^{8k-1})\#\cdots\#(S^1\times S^{8k-1})}_{k\text{ copies}}
\]
is a closed Spin $8k$-manifold with $\chi(X_k)=1$, so removing a disk gives $W$ with $\chi(W)=0$. This agrees with the existence of a normal vector field by the Poincar\'e--Hopf theorem \cite{CG17}. Similarly, $\mathbb OP^2\#(S^1\times S^{15})$ is a $16$-dimensional String manifold with Euler characteristic $1$ and gives the desired $W$ for $n=15$.
\end{remark}

\subsection{The bordism groups of Bais--May Custodio--Torres}\label{subsec:BCT-application}

We next consider the Spin pseudo-Riemannian bordism introduced by Bais, May Custodio, and Torres \cite{BCT24}. The existence of such a bordism reduces to specifying an oriented two-plane subbundle $\zeta\subset TW$ on a Spin bordism $W^{n+1}$, requiring that along the boundary $\zeta$ be the direct sum of the normal line bundle and a trivial tangent line bundle (cf.~\cite[Theorem~26 and Lemma~33]{BCT24}). We denote the corresponding bordism relation by $\sim_{\mathrm{BCT}}$ and the bordism group by $\Omega_{2,n-1}^{\Spin_0}$.

Let $n\geq2$ and consider a nowhere-zero tangent vector field $v$ on a closed Spin $n$-manifold $M$. The orthogonal complement of the oriented two-plane bundle $\langle v\rangle\oplus\mathbb R$ in $TM\oplus\mathbb R$ is a Spin$^c$ bundle of rank $n-1$. Thus, by Theorem~\ref{thm:unstable-pt}, $(M,v)$ defines an element of
\[
A_n:=\pi_{-1}\bigl(B\Spin^c(n-1)^{-(L_{n-1}^{\Spin^c}\oplus\mathcal O_{n-1})}\bigr).
\]
By Proposition~\ref{prop:direct-bordism}, $M_0\sim_{\mathrm{BCT}}M_1$ if and only if $[M_0,v_0]=[M_1,v_1]\in A_n$ for suitable tangent vector fields $v_0,v_1$.

The middle vertical column of the Smith diagram for the two-plane pattern gives the exact sequence
\[
\begin{tikzcd}[column sep=large]
\pi_{-1}(B\Spin^c(n)^{-\mathcal O_n})
\arrow[r]&A_n\arrow[r,"j"]&
\SpinPlane\arrow[r,"\rho"]&\mathbb Z
\end{tikzcd}.
\]
Here $j$ is the map induced by rank stabilization, and $\rho$ takes the Euler number of the complementary bundle.

For a closed manifold $M$ of odd dimension $n=2r+1$, define the Kervaire semi-characteristic by
\[
\widehat\chi_{\mathbb Z/2}(M)
=\sum_{i=0}^{r}\dim_{\mathbb Z/2}H_i(M;\mathbb Z/2)\pmod2.
\]

\begin{thm}[Criterion for BCT bordism]\label{thm:BCT-final}
Let $n\geq1$, and let $M_0,M_1$ be closed Spin $n$-manifolds whose connected components all have Euler characteristic zero. Then $M_0\sim_{\mathrm{BCT}}M_1$ if and only if $[M_0]=[M_1]\in\Omega_n^{\Spin}$ and, when $n\equiv3\pmod8$, the additional condition
\[
\widehat\chi_{\mathbb Z/2}(M_0)
=\widehat\chi_{\mathbb Z/2}(M_1)
\]
holds.
\end{thm}

\begin{proof}
For $n=1$, we may take $\zeta=TW$, so assume $n\geq2$. First we show that $j$ is injective. Put $X=B\Spin^c(n)$ and $V=\mathbb R^2-\mathcal O_n$, so $\pi_{-1}(X^{-\mathcal O_n})=\pi_1(X^V)$. The space $X$ is simply connected, and the differential in the Atiyah--Hirzebruch spectral sequence
\[
d_2:H_2(X;\mathbb Z)\longrightarrow H_0(X;\mathbb Z/2)
\]
is given by evaluation of $w_2(V)=c_1(\mathcal O_n)\bmod2$ and is surjective. Thus the terms of total degree $1$ vanish and $\pi_1(X^V)=0$. Hence $j$ is injective.

For a nowhere-zero tangent vector field $v$ on $M$, the two-plane bundle corresponding to $j[M,v]$ is trivial, so
\[
\epsilon(j[M,v])=(0,[M]),\qquad \rho(j[M,v])=0.
\]
Since every connected component has Euler characteristic zero, we may choose nowhere-zero tangent vector fields $v_i$ on $M_i$. Thus, if $[M_0]=[M_1]\in\Omega_n^{\Spin}$, then
\[
\Delta:=j[M_0,v_0]-j[M_1,v_1]
\in\ker\epsilon\cap\ker\rho.
\]

For $n\equiv1,5,7\pmod8$, Theorem~\ref{thm:spin-plane} gives $\ker\epsilon=0$, so $\Delta=0$. If $n$ is even, then $\ker\epsilon\cong\mathbb Z$, but
\[
\rho([S^n,\mathbb R^2 \subset TS^n\oplus\mathbb R^2])=\chi(S^n)=2,
\]
and $[S^n,\mathbb R^2 \subset TS^n\oplus\mathbb R^2]\in\ker\epsilon$. Hence $\rho|_{\ker\epsilon}$ is injective, and again $\Delta=0$.

Finally, consider $n\equiv3\pmod8$. In this case $\ker\epsilon\cong\mathbb Z/2$. By Bais--May Custodio--Torres \cite[Theorem~B]{BCT24}, the Kervaire semi-characteristic is invariant under BCT bordism. In particular, since $\widehat\chi_{\mathbb Z/2}(S^n)=1$, for a nowhere-zero tangent vector field $v$ on $S^n$, the class $j[S^n,v]\in\ker\epsilon$ is nonzero (by injectivity of $j$ and invariance of the semi-characteristic). If $\Delta=j[S^n,v]$, injectivity of $j$ implies
\[
[M_0,v_0]=[M_1,v_1]+[S^n,v]\quad\text{in }A_n,
\]
and invariance of the semi-characteristic under bordism gives
\[
\widehat\chi_{\mathbb Z/2}(M_0)
=\widehat\chi_{\mathbb Z/2}(M_1)+1.
\]
Hence if the semi-characteristics agree, then $\Delta=0$.

Thus, in every case, the required conditions imply $\Delta=0$, and injectivity of $j$ yields a BCT bordism. The converse follows from invariance of the ordinary Spin bordism class and the semi-characteristic under bordism.
\end{proof}

\begin{cor}
Let $n\geq1$. If $n$ is odd, then
\[
\Omega_{2,n-1}^{\Spin_0}\cong
\begin{cases}
\Omega_n^{\Spin}\oplus\mathbb Z/2,&n\equiv3\pmod8,\\
\Omega_n^{\Spin},&n\not\equiv3\pmod8
\end{cases}.
\]

\end{cor}


\begin{thebibliography}{20}

\bibitem{BCT24}
Valentina Bais, Victor Gustavo May Custodio, and Rafael Torres.
\newblock Existence results of $\mathrm{spin}(2,n-1)_0$-pseudo-Riemannian cobordisms.
\newblock \emph{Rev. R. Acad. Cienc. Exactas F\'is. Nat. Ser. A Mat. RACSAM}, 118(3):Paper No.~130, 2024.

\bibitem{BDS15}
Marcel B\"okstedt, Johan Louis Dupont, and Anne Marie Svane.
\newblock Cobordism obstructions to independent vector fields.
\newblock \emph{Q. J. Math.}, 66(1):13--61, 2015.

\bibitem{BS14}
Marcel B\"okstedt and Anne Marie Svane.
\newblock A geometric interpretation of the homotopy groups of the cobordism category.
\newblock \emph{Algebr. Geom. Topol.}, 14(3):1649--1676, 2014.

\bibitem{CG17}
Diarmuid Crowley and Mark Grant.
\newblock The Poincar\'e--Hopf theorem for line fields revisited.
\newblock \emph{J. Geom. Phys.}, 117:187--196, 2017.

\bibitem{Detal26}
Arun Debray, Sanath K. Devalapurkar, Cameron Krulewski, Yu Leon Liu, Natalia Pacheco-Tallaj, and Ryan Thorngren.
\newblock The Smith fiber sequence and invertible field theories.
\newblock \emph{Comm. Math. Phys.}, 407(2):Paper No.~25, 2026.

\bibitem{DebrayYu}
Arun Debray and Matthew Yu.
\newblock Adams spectral sequences for non-vector-bundle Thom spectra.
\newblock \emph{Algebr. Geom. Topol.}, to appear.
\newblock arXiv:2305.01678.

\bibitem{Ebert13}
Johannes Ebert.
\newblock A vanishing theorem for characteristic classes of odd-dimensional manifold bundles.
\newblock \emph{J. Reine Angew. Math.}, 684:1--29, 2013.

\bibitem{GTMW09}
S\o ren Galatius, Ib Madsen, Ulrike Tillmann, and Michael Weiss.
\newblock The homotopy type of the cobordism category.
\newblock \emph{Acta Math.}, 202(2):195--239, 2009.

\bibitem{HanHuang}
Fei Han and Ruizhi Huang.
\newblock On characteristic numbers of $24$ dimensional String manifolds.
\newblock \emph{Math. Z.}, 301(1):541--563, 2022.

\bibitem{Hoekzema20}
Renee S. Hoekzema.
\newblock Manifolds with odd Euler characteristic and higher orientability.
\newblock \emph{Int. Math. Res. Not. IMRN}, 2020(14):4496--4511, 2020.

\bibitem{HSV25}
Renee S. Hoekzema, Luuk Stehouwer, and Simona Vesel\'a.
\newblock SKK groups of manifolds and non-unitary invertible TQFTs.
\newblock arXiv:2504.07917, 2025.

\bibitem{HSV26}
Renee S. Hoekzema, Luuk Stehouwer, and Simona Vesel\'a.
\newblock The Euler characteristic of $\mathrm{Pin}^{\pm}$ manifolds.
\newblock arXiv:2610.06747, 2026.

\bibitem{Koschorke06}
Ulrich Koschorke.
\newblock Nonstabilized Nielsen coincidence invariants and Hopf--Ganea homomorphisms.
\newblock \emph{Geom. Topol.}, 10:619--666, 2006.

\bibitem{McTague14}
Carl McTague.
\newblock The Cayley plane and String bordism.
\newblock \emph{Geom. Topol.}, 18(4):2045--2078, 2014.

\bibitem{Ochanine81}
S. Ochanine.
\newblock Signature modulo $16$, invariants de Kervaire g\'en\'eralis\'es et nombres caract\'eristiques dans la $K$-th\'eorie r\'eelle.
\newblock \emph{M\'em. Soc. Math. France (N.S.)}, 5, 1981.

\bibitem{Reinhart63}
Bruce L. Reinhart.
\newblock Cobordism and the Euler number.
\newblock \emph{Topology}, 2:173--177, 1963.

\bibitem{Rudyak98}
Yuli B. Rudyak.
\newblock \emph{On Thom Spectra, Orientability, and Cobordism}.
\newblock Springer Monographs in Mathematics. Springer-Verlag, 1998.

\bibitem{Thom54}
Ren\'e Thom.
\newblock Quelques propri\'et\'es globales des vari\'et\'es diff\'erentiables.
\newblock \emph{Comment. Math. Helv.}, 28:17--86, 1954.

\end{thebibliography}
\end{document}